\documentclass[12pt]{article}

\usepackage[english]{babel}

\usepackage[a4paper,top=2cm,bottom=2cm,left=2cm,right=2cm,marginparwidth=1.75cm]{geometry}
\usepackage[colorlinks=true, allcolors=blue]{hyperref}
\usepackage{natbib}
\usepackage[dvips]{graphicx}
\usepackage{amsmath,amsfonts,amssymb,latexsym,epsfig}
\usepackage{mathrsfs,mdframed, float,graphicx,dsfont,wrapfig}
\usepackage{mathrsfs}
\usepackage{verbatim}
\usepackage{latexsym}
\usepackage{amsthm}
\usepackage{amssymb}
\usepackage{graphics}
\usepackage{amsbsy}
\usepackage{enumerate}
\usepackage{times}
\usepackage{mathtools}   
\usepackage{bm}
\usepackage{sidecap}
\usepackage{cleveref}
\usepackage{url}

\newtheorem{theorem}{Theorem}[section]
\newtheorem{lemma}[theorem]{Lemma}
\newtheorem{remark}[theorem]{Remark}

\newtheorem{assumption}{Assumption}

\newtheorem{proposition}{Proposition}
\newtheorem{definition}{Definition}

\usepackage{amsmath}
\usepackage{amsfonts} 
\usepackage{graphicx}
\usepackage[colorlinks=true, allcolors=blue]{hyperref}
\usepackage{authblk}

\newcommand{\R}[1]{\mathbb{R}}

\definecolor{brinkpink}{rgb}{0.98, 0.38, 0.5}
\definecolor{blue_dark}{rgb}{0.0, 0.5, 0.69}

\newcommand{\ds}[1]{\textcolor{black}{#1}}

\title{Posterior concentration in spatio-temporal \\ Hawkes processes}

\author[1]{Xenia Miscouridou} 
\author[2]{D\'eborah Sulem} 

\affil[1]{Department of Mathematics and Statistics, University of Cyprus\footnote{miscouridou.xenia@ucy.ac.cy}} 
\affil[2]{Faculty of Informatics, Universit\'a della Svizzera italiana\footnote{deborah.sulem@usi.ch}} 

\date{}

\begin{document}

\newcommand{\fa}[1]{{\color{blue} #1}}
\maketitle

\begin{abstract}
We develop a Bayesian nonparametric framework for inference in multivariate spatio-temporal Hawkes processes, extending existing theoretical results beyond the purely temporal setting. Our framework encompasses modelling both the background and triggering components of the Hawkes process through Gaussian process priors. Under appropriate smoothness and regularity assumptions on the true parameter and the nonparametric prior family, we derive posterior contraction rates for the intensity function and the parameter, in the asymptotic regime of repeatedly observed sequences. Our analysis generalizes known contraction results for purely temporal Hawkes processes to the spatio-temporal setting, which allows to jointly model excitation and clustering effects  across time and space. These results provide, to our knowledge, the first theoretical guarantees for Bayesian nonparametric methods in spatio-temporal point data.
\end{abstract}


\tableofcontents

\vfill
\section{Introduction}
Hawkes processes are point process models designed to capture sequences of events where the intensity of occurrence depends on the past history of the process. Their defining property is self-excitation: each past event increases the likelihood of future ones. Originally introduced by Alan G. Hawkes (1971) for modeling the clustering of earthquakes~\citep{hawkes1971spectra}, Hawkes processes have since been widely applied across disciplines, including seismology~\citep{ogata1988statistical}, social and information networks~\citep{crane2008robust, zhao2015seismic}, neuroscience~\citep{reynaud2014goodness, truccolo2005point}, dynamic network analysis~\citep{xu2016learning, eichler2017graphical}, criminology~\citep{mohler2011self, miscouridou2023cox}, and epidemiology~\citep{rizoiu2018sir}. More recently, Hawkes processes have also found applications in machine learning, where they are used to model temporal dependencies, perform causal discovery, and augment large language models with event-based dynamics and memory~\citep{mei2017neural, zuo2020transformer, huang2024large, hills-etal-2024-exciting}.

A Hawkes process can be viewed as a non-homogeneous cluster Poisson point process and admits a self-exciting intensity function. It can also be represented as a branching or cluster process with a latent structure. This representation is particularly useful for simulation and interpretation, as the process can be viewed as a cascade of events, where each event is either exogeneously generated or endogeneously generated by a past event (parent). Hawkes processes can be univariate or multivariate. In the latter case, each component corresponds to a distinct type of event, and the process is equivalent to a marked point processes.

Originally, a Hawkes process was defined as a univariate temporal point process \citep{hawkes1971spectra}, where each event at time $t_i$ increases the probability of future events at times $t > t_i$. This temporal process does not allow to model spatial effects in spatio-temporal event data, or, in other words, it assumes that the influence of an event is homogeneous across space, which is often unrealistic in many applications. In practice, the excitation effect caused by an event may depend on both time and spatial proximity, making a spatio-temporal formulation more appropriate. Indeed, recent studies have emphasized spatio-temporal Hawkes processes in applications such as modeling wildfires~\citep{koh2023spatiotemporal} and terrorism~\citep{jun2024flexible}. Comprehensive overviews can be found in~\cite{reinhart2018review} and more recently in~\cite{bernabeu2025spatio}. Despite these advances, there remains a significant gap between practical modeling approaches and theoretical understanding, particularly from a Bayesian perspective.

From a Bayesian viewpoint, establishing posterior contraction rates provides fundamental theoretical validation for a model’s ability to learn the true self-exciting mechanism as the amount of data increases. Existing work on posterior contraction for Hawkes processes has focused almost exclusively on temporal models.~\citet{donnet2020nonparametric} derived posterior contraction rates for multivariate linear Hawkes processes in a nonparametric setting, while \citet{sulem2024bayesian} extended the analysis to nonlinear Hawkes processes, accounting for inhibition effects. More general mathematical frameworks for posterior contraction in point processes are provided in~\citet{donnet2014posterior}, and analogous results for inhomogeneous Poisson processes and using Gaussian process priors can be found in~\citet{Kirichenko_vanzanten15} and~\citet{GiordanoKirichenkoRousseau2025}.

On the frequentist side, likelihood-based inference for Hawkes models has a long history. \citet{ogata1978asymptotic} and \citet{ozaki1979maximum} established consistency and asymptotic normality of the maximum likelihood estimator (MLE) for stationary, univariate, exponential and purely temporal Hawkes processes, while \citet{liniger2009multivariate} extended these results to the multivariate case. For the purely temporal but nonstationary Hawkes model, \cite{chen2013inference} and \cite{kwan2023alternative} study the consistency of the MLE in an asymptotic setting closely related to ours. Broader results for MLE estimation in point processes can be found in Chapter 7 of~\citet{daley2003pointprocessesII}. Recent work has also derived non-asymptotic, finite-sample concentration inequalities for least-square estimation in multivariate temporal Hawkes processes, both in parametric and nonparametric settings~\citep{clinet2017inference, hansen2015lasso, cai2022latent}.

However, none of the existing Bayesian posterior contraction results or frequentist asymptotic results address the spatio-temporal setting, to our best knowledge. Theoretical guarantees for Bayesian inference in spatio-temporal Hawkes processes remain unexplored, despite their growing empirical importance. Partial advances have been made only recently, such as the flexible spatio-temporal modeling framework in~\citet{siviero2024flexible}, but without asymptotic or contraction results. This gap motivates the present work, which provides a rigorous Bayesian nonparametric treatment of spatio-temporal, non-stationary Hawkes processes and establishes their posterior contraction properties.

To study these types of theoretical guarantees, different asymptotic setups are possible, such as repeated observations~\citep{dolmeta2025nonparametric} or infinite domain~\citep{GiordanoKirichenkoRousseau2025}, and for each of these different Bernstein-type inequalities are needed. 
\paragraph{Contribution.} 

We establish posterior contraction rates for  non-stationary and spatio-temporal Hawkes processes within a Bayesian nonparametric framework in the setting of repeated observations. The nonparametric framework permits a flexible specification of the conditional intensity of the Hawkes process, using nonparametric prior families over functions of space and time for both the background rate and the triggering kernel. Under suitable regularity assumptions on the true parameter and mild conditions on the prior family, we derive explicit rates at which the posterior distribution concentrates around the truth. Our results hold for general classes of nonparametric priors, in particular encompassing Gaussian process priors, which provide a natural and widely used choice in modern Bayesian inference for point processes \citep{Zhang_2020,lloyd2015variational, malem2022variational}. Our proofs have similar structure to those of papers of temporal Hawkes~\citep{donnet2020nonparametric,sulem2024bayesian} but the extension to space-time is non-trivial and requires new concentration inequalities. Additionally, our work differs from the majority of previous point process papers as we consider the repeated observations settings rather than an infinite domain. Our analysis therefore extends existing posterior contraction results for temporal Hawkes to the general spatio-temporal and non-stationary setting.

The rest of the paper is organized as follows. Section~\ref{sec:setup} gives the setup and introduces the multivariate spatio-temporal Hawkes process model and the Bayesian nonparametric formulation illustrated with Gaussian process priors on the background and triggering components. Section~\ref{sec:concentration} presents the main theoretical results, establishing posterior contraction rates under suitable regularity conditions on the true intensity and the prior. 
Proofs of all main results can be found in Section~\ref{sec:proofs}. 



\section{Setup and Methodology}
\label{sec:setup}
\subsection{Setup}
We assume that we have repeated observations of a point process over a bounded spatio-temporal domain $S$. For simplicity we can assume $S = [0,1] \times [0,1]^d$ (note that we can always rescale the events time to $[0,1]$)\footnote{We note that our methodology can easily be modified for general bounded domains by changing the support of the parameters.}. In many practical applications, $d=2$ (latitude and longitude). The data thus consists of $n$ i.i.d. sequences of events with spatio-temporal coordinates, i.e., sequence $i$ is a set of $m_i$ points $N^i = (t^i_j, s^i_j)_{j \leq m_i}$ with $s_j^i \in \mathbb R^d$ and
\begin{align*}
    t^i_1 < t^i_2 < \dots <  t^i_{m_i}
\end{align*}

We denote by $N = (N^i)_{i\leq n}$ these sequences and model the latter as independent realisations of the same spatio-temporal Hawkes process $N(t,s)$ defined on $S$ as follows.

\begin{definition}\label{def:st-hawkes}
    A spatio-temporal point process $N(t,s)$ defined on a  domain $S$ is a spatio-temporal Hawkes process with parameter $f = (\mu, g)$ where $\mu \geq 0$ and $g \geq 0$ are non-negative functions, respectively called the background rate and the triggering kernel, if for any $(t,s) \in S$, its conditional intensity function is
\begin{align*}
    \lambda_{t,s}(f) = \lambda_{t,s}(\mu, g) &= \mu(t,s) + \int_{[0,t) \times [0,1]^d} g(t-t',s-s')dN(t',s') \\
    &= \mu(t,s) + \sum_{(t_j, s_j) \in N, t_j < t} g(t-t_j,s - s_j),
\end{align*}

\end{definition}

 Note that a Hawkes process as defined in Definition \ref{def:st-hawkes} is non-stationary unless $\mu$ is constant in time.

 We denote by $\mathbf{P}_{f}$ the law of the Hawkes process $N(t,s)$ with parameter $f$ and $\mathbb E_f$ the corresponding expectation. For a subset $A \subset S$, we denote by $N(A)$ the number of observations on $A$. We also make a finite-range assumption on the triggering kernel $g$, namely $g(t,s) = 0$  if $t<0$ or $t>a$ or $\|s\|_\infty > b$ with $0 < a<1/2, 0<b<1/2$. This implies that we can re-write the intensity as 
\begin{align}
     \lambda_{t,s}(f)  &= \mu(t,s) + \int_{t-a}^{t-} \int_{s' \in [0,1]^d :\|s - s'\|_\infty \leq b} g(t-t',s-s')dN(t',s').
    \label{eq:cond_int}
\end{align}
We make another standard assumption that the branching ratio of $N(t,s)$ is less than 1, implying that the process is non-explosive, i.e.,
\begin{align*}
    \|g\|_1 :=  \int_0^a  \int_{[0,1]^d \cap \{s: \|s\|_\infty\leq b\} }g(t,s) dt ds < 1.
\end{align*}


Here the statistical goal is to estimate $f$ from observations $N$. We first prove an identifiability result, which validates the feasibility of this estimation problem, under a mild assumption on the background rate.

\begin{assumption}\label{ass:mu}
    The background rate $\mu$ verifies:
    \begin{enumerate}
        \item $\int_0^{1-a} \int_{[b,1-b]^d}\mu(t,s)>0$.
        \item $\mu(t,s) < +\infty, \quad \forall (t,s) \in S$.
    \end{enumerate}
\end{assumption}

Assumption \ref{ass:mu} ensures that the background rate is finite and that the probability of observing at least one event is non-null.

\begin{proposition}[Identifiability]\label{lem:identifiability}
    Let $N$ and $N'$ be two spatio-temporal Hawkes processes with respective parameters $f = (\mu, g)$ and $f' = (\mu',g')$ verifying Assumption \ref{ass:mu}. Then,
    \begin{align*}
        N \overset{d}{=} N' \iff f = f'.
    \end{align*}
\end{proposition}
The proof of Proposition~\ref{lem:identifiability} is found in Section~\ref{sec:proofs}.

\subsection{Methodology}\label{sec:methods}

 We denote by $f_0=(\mu_0,g_0)$ the true parameter and by $\mathbf{P}_{0}$ and $\mathbb E_0$ the law of the Hawkes process $N(t,s)$ and its expectation respectively.
 
We now describe our Bayesian nonparametric estimation framework for  the true parameter $f_0$ of the spatio-temporal Hawkes process (Definition~\ref{def:st-hawkes}). Here, we focus on  prior families based on transformations of Gaussian processes (GP), though our theorerical results hold (Section~\ref{sec:concentration}) for more general families such as mixture of beta densities or histogram priors (see, e.g.,~\cite{donnet2020nonparametric} and~\cite{sulem2024bayesian}). Recall that our parameter of interest $f = (\mu, g) \in \mathcal{F}$ where
\begin{align*}
    \mathcal{F} = \{ f = (\mu, g) ; \mu : S \to \mathbb R_+, \: g:[0,a] \times [0,b]^d \to \mathbb R_+ \}.
\end{align*}
We define a prior distribution $\Pi$ on $\mathcal{F}$ which factorises over the background and triggering kernel, i.e.,
\begin{align*}
    \Pi(f) = \Pi_\mu(\mu)  \Pi_g(g), \qquad f \in \mathcal{F}.
\end{align*}
The prior distributions  $\Pi_\mu, \Pi_g$ are distributions on non-negative functions implicitly constructed via transformations of GP. Specifically, 
\begin{align*}
    &\mu = \sigma(\nu), \qquad \nu \sim GP( 0, k_\nu)  \\
    &g = \sigma(\phi), \qquad \phi \sim GP( 0, k_\phi).
\end{align*}
Above, $\nu$ and $\phi$ are latent functions and  $\sigma: \mathbb R  \rightarrow \mathbb R^+$ is a known link function, typically a strictly increasing and bijective function on a large enough interval such as the softplus or the sigmoid function. Moreover, $k_\nu$ and $k_\phi$ are covariance functions (kernels) defined on the spatio-temporal domain $S$. For simplicity and without loss of generality, we choose a zero mean function in our GP prior.

GP priors are commonly used in Bayesian nonparametric methods for point processes, e.g., for inhomogeneous Poisson processes~\citep{adams2009tractable, lloyd2015variational, kirichenko2015optimality,  palacios2013gaussian, GiordanoKirichenkoRousseau2025, ng2019estimation} as well as temporal Hawkes processes~\citep{Zhang_2020, malem2022variational}. In~\cite{miscouridou2023cox}, a GP prior with exponential link function is used for estimating the spatio-temporal background rate of a Hawkes process. In constrast, here, both the background and the triggering kernel are estimated nonparametrically using GP priors. For an introduction to GPs, see, e.g.~\cite{rasmussen2005gaussian}.

In the rest of this section, we specify possible choices for the kernel functions and inference methodology. Let $u = (t,s) \in S$. Common choices of kernels include the squared exponential (RBF) and Mat\'ern kernels defined as follows:
\begin{align}
    &k_{RBF}(u,u') = \sigma  \exp \left( - \frac{\|u-u'\|^2}{\ell^2} \right), \label{eq:rbf} \\
     &k_{Mat}(u,u') = \frac{\sigma^2}{\Gamma(\tau) 2^{\tau-1}} \left( \frac{\|u-u'\|_2}{\ell} \right)^{\tau} B_\nu \left( \frac{\|u - u'\|_2}{\ell}\right), \label{eq:matern}
\end{align}
with hyperparameters $\sigma^2, \ell, \tau > 0$, $\Gamma$ the Gamma function, and $B_\tau$ the modified Bessel function of the second kind. We note that in the limit $\tau \to \infty$, the Mat\'ern kernel is equivalent to the RBF kernel. Often, it is computationally convenient to use kernels that are separable in time and space and stationary, i.e.,
\begin{align*}
    &k_r(u,u') = k_{r,t}(|t - t'|)k_{r,s}(\|s - s'\|_2), \quad r \in \{\nu, \phi\}.
\end{align*}
Here as well, the RBF and Mat\'ern kernels are common choices for the temporal and spatial kernels $k_{r,t}$ and $k_{r,s}$. We note that a separable kernel does not imply that the latent function  $\nu$ or $\phi$ (and a-fortiori $\mu$ or $g$) are separable functions over space and time. Nonetheless, the choice of kernel and its hyperparameters determines the smoothness of the samples from the GP prior (see more on this in Section \ref{sec:concentration}).


Inference on $\mu$ and $g$ is then performed via the posterior distribution. Firstly, we define the likelihood of the set of observations $N = (N^i)_{i=1, \dots, n}$ as  (see, e.g., \cite{daley2003pointprocessesII})
\begin{align*}
    L(N | f) = \prod_{i=1}^n \prod_{j=1}^{m_i} \lambda^i_{t_j^i, s_{j}^i}(f) \exp \left \{ - \int_S \lambda^i_{t,s}(f) dt ds \right \}, \qquad f \in \mathcal{F},
\end{align*}
where, for each $i=1, \dots, n$,
\begin{align}\label{eq:lambda-i}
    \lambda^i_{t, s}(f) = \mu(t,s) + \int g(t-t',s-s')dN^i(t',s').
\end{align}
The posterior distribution is then defined as
\begin{align*}
    \Pi(B|N) = \frac{\int_{B} L(N | f)d\Pi(f)}{\int_{\mathcal{F}} L(N | f)d\Pi(f)} , \qquad B \subset \mathcal{F}.
\end{align*}

In practice, a variational approximation of the posterior may only be computed, defined, e.g., as
\begin{align}\label{eq:var-post}
    \hat Q(f) = \arg \min_{Q \in \mathcal{Q}} KL(Q|| \Pi(f|N)),
\end{align}
where $KL$ is the Kullback-Leibler divergence. Here, the minimum is taken over an approximating family of distributions $\mathcal{Q}$, for instance, Gaussian processes on $S$. This approach is used by \cite{lloyd2015variational} and \cite{Zhang_2020, zhou2020, sulem2022scalable} respectively in the context of Poisson and temporal Hawkes processes. In fact, since the posterior  is non-conjugate here, sampling from the posterior using Monte-Carlo Markov Chain techniques is notoriously intensive. 
Note that finding the minimiser in \eqref{eq:var-post} is equivalent to maximising the Evidence Lower Bound (ELBO) defined as
\begin{align*}
    ELBO(Q) =  \mathbb E_Q[\log (L(N|f)\Pi(f))] - \mathbb E_Q[\log Q(f)].
\end{align*}

\section{Posterior concentration} 
\label{sec:concentration}

\subsection{General results}

In this section we analyse the asymptotic properties of the posterior distribution 
as the number of observed sequences $n \to \infty$. Precisely, we establish general concentration rates for the posterior on the intensity function $\lambda_{t,s}(f)$ and on the parameter $f$. For this, we first define the stochastic  distance ($L_1$-distance on the intensity function) between any pair of parameters $f,f' \in \mathcal{F}$ as below
\begin{align*}
    d_S(f,f') := \frac{1}{n} \sum_{i=1}^n \int_{[0,1]} \int_{[0,1]^d} |\lambda^i_{t,s}(f) -  \lambda^i_{t,s}(f')| dt ds = \frac{1}{n} \sum_{i=1}^n \|\lambda^i(f) -  \lambda^i(f')\|_1,
\end{align*}
where $\lambda^i_{t,s}(f)$ is defined as in \eqref{eq:lambda-i} and $\lambda^i(f)$ denotes the corresponding function from $S$ to $\mathbb R^+$.

Then we define the $L_1$-distance on the parameter as
\begin{align*}
    \|f - f'\|_1 := \|\mu - \mu'\|_1 + \|g-g'\|_1, \qquad f,f' \in \mathcal{F}.
\end{align*}

Our first result is the contraction of $\Pi(f|N)$ on the true parameter $f_0$ in terms of the stochastic distance, i.e., 
\begin{align}\label{eq:conc-sd}
    \Pi( d_S(f,f_0) > M  \epsilon_n | N_n) \xrightarrow[n \to \infty]{\mathbb P_{0}} 0. 
\end{align}
where $\epsilon_n = o(1)$ is called the contraction rate and $M>0$ is an arbitrarily large constant. Before formally stating our result, we state our assumptions on $f_0$, $\bar \epsilon_n$ and the prior. \ds{To include in our theory the case of RBF kernel (see Section \ref{sec:GP}), we formulate our assumption in terms of two sequences $\epsilon_n, \bar \epsilon_n = o(1)$ such that $\epsilon_n \geq \bar \epsilon_n$.}

\begin{assumption}[Bounded parameter]\label{ass:boundedness}
Recall that $f_0 = (\mu_0, g_0)$. We assume that $\|g_0\|_1 < 1$ and there exist $\underline{\mu}, \bar \mu, \bar g > 0$ constants independent of $n$ such that for each $(t,s) \in S$,
\begin{align*}
    &\underline{\mu} \leq \mu_0(t,s) \leq  \bar \mu \\
    &0 \leq g_0(t,s) \leq \bar g.
\end{align*}
\end{assumption}

\begin{assumption}[Prior mass]\label{ass:prior-mass}
Let
\begin{align*}
    B_\infty(\bar \epsilon_n) = \{ f = (\mu, g)  ; \|\mu - \mu_0\|_\infty  + \|g - g_0\|_\infty \leq \bar \epsilon_n\}.
\end{align*}
There exists $c_1 > 0$ such that  $\Pi(B_\infty(\bar \epsilon_n)) \geq e^{- c_1 n \bar \epsilon_n^2}$. 
\end{assumption}

\begin{assumption}[Sieves]\label{ass:sieves}
Let $\Lambda_{0,2} := \mathbb E_0\left[ \int  (\lambda^1_{t,s}(f_0))^2  dt ds \right]$ and
\begin{align}\label{eq:def-kappa}
    \kappa := \frac{4 \log 2}{\underline{\mu}}\left \{ 2 + 4\left(\frac{\bar \mu}{1 - \|g_0\|_1} + \Lambda_{0,2} \right) \right \},
\end{align}
where $\underline{\mu},\bar \mu$ are defined in Assumption \ref{ass:boundedness}. There exist $\mathcal{F}_n \subset \mathcal{F}$ and $ c_2 > c_1 + \kappa $, $\zeta_0 > 0$ and $c_3 > 0$ constants independent of $n$ such that $\Pi(\mathcal{F}_n) \geq 1 - e^{-c_2 n \bar \epsilon_n^2}$ and
    \begin{align*}
        C(\zeta_0 \epsilon_n, \mathcal{F}_n, \|.\|_1) \leq e^{c_3 n \epsilon_n^2}.
    \end{align*}
where $C(\epsilon_n, \mathcal{F}_n, \|.\|_1)$ is the covering number of $\mathcal{F}_n$ with balls of radius $\epsilon_n$ in terms of $L_1$-norm.
\end{assumption}

Assumption \ref{ass:prior-mass} and \ref{ass:sieves} resemble those in \cite{donnet2020nonparametric, GiordanoKirichenkoRousseau2025, sulem2024bayesian}. In Section \ref{sec:GP}, we show that those assumptions are verified for our GP-based prior under mild conditions on the kernel function. Assumption \ref{ass:prior-mass} is a boundedness assumption on the true parameter $f_0$ which is not restrictive in practice. Similar upper bounds are commonly assumed in the literature on point processes, see e.g., \cite{GiordanoKirichenkoRousseau2025}. The lower bound ensures that the probability of an event is non-null at any point $(t,s) \in S$.

%

\begin{proposition}[Concentration in stochastic distance]\label{prop:conc-ds}
    Under Assumptions \ref{ass:boundedness}, \ref{ass:prior-mass}, and \ref{ass:sieves}, and if $n \bar \epsilon_n^2 \to \infty$ \ds{and $\epsilon_n = o((\log n)^{-2})$ }, then \eqref{eq:conc-sd} holds.
\end{proposition}

While posterior concentration in stochastic distance gives prediction guarantees, it is a non-explicit distance on the parameter space. Therefore, to obtain guarantees on parameter interpretation (e.g., how much the endogeous/exogeneous effects are in the event generating process), we establish a second result which is the posterior concentration rate in terms of the $L_1$-distance on $\mathcal{F}$.

\begin{proposition}[Concentration in $L_1$-distance]\label{prop:conc-l1}
    Under  Assumptions \ref{ass:boundedness}, \ref{ass:prior-mass}, \ref{ass:sieves},   and if $n \bar \epsilon_n^2 \to \infty$ \ds{and $\epsilon_n = o((\log n)^{-2})$ }, then 
    \begin{align}
    \label{eq:conc-L1}
    \Pi( \|f - f_0\|_1 > M' \epsilon_n | N_n) \xrightarrow[n \to \infty]{\mathbb P_{0}} 0, 
\end{align}
with $M'>0$ an arbitrarily large constant.
\end{proposition}

The proofs of Propositions \ref{prop:conc-ds} and \ref{prop:conc-l1} are reported in Section \ref{sec:proofs}. In the next section, we show an application of our result to our GP-prior (see Section \ref{sec:GP}) and to H\"older classes of functions.


\subsection{Application to Gaussian process priors and H\"older smooth functions}\label{sec:GP}

Recall our GP-based prior construction from Section \ref{sec:methods} with latent functions $\nu,\phi$, link function $\sigma$ and covariance functions (kernels) $k_\nu, k_\phi$. We demonstrate that under mild assumptions on $\sigma$ and the kernels that Assumptions \ref{ass:prior-mass} and \ref{ass:sieves} are verified, and the concentration rate $\epsilon_n$ is explicit in the smoothness of the true parameter $f_0$.  Due to their popularity in practical applications, we focus in this section on the M\'atern and the squared exponential kernels (defined in \eqref{eq:rbf}, \eqref{eq:matern}).

Before stating our assumptions, we introduce some notation. For $\alpha > 0$, let $C^\alpha(S)$ be the space of H\"older $\alpha$-smooth functions, i.e., functions which are $\lfloor \alpha \rfloor$-times differentiable and which $\lfloor \alpha \rfloor$-th derivative is $(\alpha - \lfloor \alpha \rfloor)$-continuous, i.e., for $f \in C^\alpha(S)$,
\begin{align*}
   | f^{\lfloor \alpha \rfloor}(x) - f^{\lfloor \alpha \rfloor}(y)| \leq |x-y|^{\alpha - \lfloor \alpha \rfloor}.
\end{align*}
For $\alpha \in \mathbb N$ we denote by $\mathcal{S}^\alpha(S)$ the Sobolev space of order $\alpha$, i.e., functions which (weak) derivatives $D^\gamma f$ are squared integrable for any $\|\gamma\|_1 \leq \alpha$. Sobolev spaces of order $\alpha > 0$ can also be defined via the Fourier transform, see e.g., Definition C.6 in \cite{ghosal2017fundamentals}.

Our first two assumptions are mild regularity and smoothness conditions on $\sigma$ and $(\mu_0, g_0)$.

\begin{assumption}\label{ass:link}
    The link function $\sigma: \mathbb R \to \mathbb R^+$ is infinitely smooth, strictly increasing and $L$-Lipschitz with $L>0$. Moreover, it is bijective from $\mathbb R$ to $(0,C)$ with $C> \bar \mu \vee \bar g$ and $\bar \mu, \bar g$ defined in Assumption \ref{ass:boundedness}.
\end{assumption}

\begin{assumption}\label{ass-f0}
    The functions $\mu_0$ and $g_0$ are H\"older $\alpha$-smooth with $\alpha > 0$, i.e., $\mu_0, g_0 \in C^\alpha(S)$. 
\end{assumption}

\begin{remark}
    \ds{The softplus and the scaled sigmoid function $\sigma(t) = \sigma^* (1 + e^{-\alpha^* t})^{-1} $ with $\alpha^*  > 0$ and $\sigma^* > \bar \mu \vee \bar g$ verify Assumption \ref{ass:link}. In fact, the commonly-used exponential function could be also be employed in our framework since it can easily be proven that our results still hold if we relax the Lipschitz assumption to a \emph{locally}-Lipschitz constraint as in \cite{dolmeta2025gaussian}, under the setting of bounded parameter (Assumption \ref{ass:boundedness}).}  
\end{remark}


For GP priors, the contraction rate depends on the smoothness of the process's sample paths and the Reproducing Kernel Hilbert Space (RKHS) associated to the kernel function (see, e.g., Chapter 11 of \cite{ghosal2017fundamentals} for more details).

For the Mat\'ern kernel on $S$ (with dimension $d+1$) with parameter $\tau > \frac{d+1}{2}$ , the sample paths are H\"older $\gamma$-smooth with $\gamma < \tau - \frac{d+1}{2}$ and the corresponding RKHS is $\mathcal{S}^{\tau + \frac{d}{2}}$. We prove in the following proposition that under the latter kernel, the posterior distribution concentrates at the rate $n^{-\frac{\tau}{2\tau + d +1}}$ if $\tau \leq \alpha$ and it corresponds to the optimal rate if $\tau = \alpha$.

\begin{proposition}[Mat\'ern covariance kernel]\label{prop:GP1}
    Under  Assumptions \ref{ass:boundedness}, \ref{ass:link}, \ref{ass-f0} and the GP-based prior with Mat\'ern kernel with parameter $\tau < \alpha$, then Assumptions \ref{ass:prior-mass} and \ref{ass:sieves} are verified  with $\epsilon_n = \bar \epsilon_n \asymp n^{-\frac{\tau}{2\tau+d+1}}$ and \eqref{eq:conc-sd} and \eqref{eq:conc-L1} hold.
\end{proposition}

\begin{remark}
    \ds{Note that the above result is non-adaptive. Obtaining adaptive results for a GP prior with the Mat\'ern kernel is particularly difficult.}
\end{remark}

\ds{In contrast, the squared-exponential covariance kernel together with an Inverse Gamma-hyperprior on the length-scale $\ell^{d+1}$ achieves adaptive and optimal estimation of the functional parameter, up to log-factors.} We note that for the squared-exponential kernel, the sample paths are analytical functions for any length scale $\ell$ and the RKHS has a more complex definition (see Lemma 11.35 in \cite{ghosal2017fundamentals}). 



\begin{proposition}[Squared-exponential covariance kernel] \label{prop:GP2}
    Under  Assumptions \ref{ass:boundedness}, \ref{ass:link}, \ref{ass-f0} and the GP-based prior with squared-exponential kernel where the length-scale parameter is \emph{a-priori} distributed according to a Inverse-Gamma distribution, i.e.,
    \begin{align*}
        \ell^{d+1} \sim IG(a_0,b_0)
    \end{align*}
    with $a_0,b_0>0$, then Assumptions \ref{ass:prior-mass} and \ref{ass:sieves} are verified  with
    \begin{align*}
        &\bar \epsilon_n \asymp (\log n)^{\frac{d+1}{2+ (d+1)/\alpha}} n^{-\frac{\alpha}{2\alpha + d +1}}, \qquad \epsilon_n \asymp (\log n)^{\frac{d+1}{2}} \bar \epsilon_n 
    \end{align*}
    and \eqref{eq:conc-sd} and \eqref{eq:conc-L1} hold.

\end{proposition}



\begin{remark}[Separable kernel]
\ds{add references to papers using separable kernels in GP}
    In spatio-temporal data, it is common to choose a kernel function which factorises over temporal and spatial variables as it leads to computational acceleration. For instance, one can use a separable Mat\'ern kernel with hyperparameters $(\tau_t,\tau_s)$. The corresponding RKHS is the tensor product Sobolev space $\mathcal{S}^{\tau_t + \frac{1}{2}}([0,1]) \otimes \mathcal{S}^{\tau_s + \frac{d}{2}}([0,1]^d)$. Recall the definitions of 
    the tensor product Sobolev spaces:
\begin{align*}
    &\mathcal{S}^{\gamma}([0,1]) \otimes \mathcal{S}^{\gamma}([0,1]^{d}) = \{ f \in L_2([0,1]^{d+1}) : D^\alpha f \in L_2([0,1]^{d+1}), \quad \forall \alpha,  \|\alpha\|_\infty \leq \gamma \}.
\end{align*}
In contrast to the Sobolev space $\mathcal{S}^{\gamma}([0,1]^{d+1})$ which requires that the (partial) derivatives of order $\zeta$ are squared integrable for each $\|\zeta\|_1 \leq \gamma$, the tensor product Sobolev space $\mathcal{S}^{\gamma}([0,1]) \otimes \mathcal{S}^{\gamma}([0,1]^{d})$ requires that partial derivatives of order $\zeta$ are squared integrable for each $\|\zeta\|_\infty \leq \gamma$, which is a strictly stronger condition. 
Therefore,  $\mathcal{S}^{\gamma}([0,1]) \otimes \mathcal{S}^{\gamma}([0,1]^{d}) \subset \mathcal{S}^{\gamma}([0,1]^{d+1})$ \citep{zhang2023regression}. In fact, $\mathcal{S}^{\gamma}([0,1]) \otimes \mathcal{S}^{\gamma}([0,1]^{d})$ contains functions that are finite linear combinations of product of functions in $\mathcal{S}^{\gamma}([0,1])$ and $\mathcal{S}^{\gamma}([0,1]^{d})$, i.e.,
\begin{align*}
    f \in H^{\gamma}([0,1]) \otimes H^{\gamma}([0,1]^{d}) \iff f = \sum_j f_{tj}(t)f_{sj}(s).
\end{align*}
In other words, $f \in \mathcal{S}^{\gamma}([0,1]) \otimes \mathcal{S}^{\gamma}([0,1]^{d})$ has an additive form but  is not in general separable in time and space. 
\end{remark}



\section{Proofs}\label{sec:proofs}

\subsection{Proof of Proposition \ref{lem:identifiability}}

First, we recall that $N,N'$ are two Hawkes processes respectively with parameter $f$ and $f'$ and that $N\overset{d}{=}N'$ if and only if $\lambda(t,s) \overset{d}{=} \lambda'(t,s)$ for almost every $(t,s)$. We also note that $f=f'$ directly implies that $ \lambda(t,s) \overset{d}{=} \lambda'(t,s), \forall (t,s)$ therefore it is sufficient to prove the reverse implication. Second, we notice that for any $t \geq 0$, with $A_t = [0,t] \times [0,1]^d $, under Assumption \ref{ass:mu},
$$ \mathbb{P}\Big( N (A_t) = 0 \Big) = e^{-\int_{0}^t \int_{[0,1]^d} \mu(u, s)\ du\ ds} > 0,$$
and similarly for $N'$.
Therefore, for any $(t,s)$, conditionally on $ N (A_t) = 0 $, we have
\begin{align*}
     \lambda(t, s) = \mu(t,s).
\end{align*}
and similarly for $N'$. Therefore, $N\overset{d}{=}N'$ implies that
\begin{align*}
     \lambda(t, s)\ |\  N (A_t) = 0 \overset{d}{=}
  \lambda'(s, t)\ |\  N' (A_t) = 0 ,
\end{align*}
i.e., $\mu(t,s ) = \mu'( t,s)$ for all $(t,s)$, which is equivalent to
$$\mu= \mu'.$$

Now since by Assumption \ref{ass:mu}, $\int_0^{1-a} \int_{[b,1-b]^d}\mu(t,s)dt ds >0 $, the probability of the event $\{N(\bar S) \geq 1\}$ with $\bar S = [0,1-a]\times [b,1-b]^d$ is non-null (and similarly for the event $\{N'(\bar S) \geq 1\}$). Conditioning on $\{N(\bar S) \geq 1\}$, we denote by $(T_1,S_1)$ (resp. $(T_1', S_1')$) the spatio-temporal coordinates of the first event of $N$ (resp. $N'$) . Then, conditionally on the event $\{N(\bar S) \geq 1 \}$, for any $t \geq T_1$ and $s \in S$,
\begin{align*}
    \lambda(t, s) = \mu(t,s) + g(t-T_1,s - S_1).
\end{align*}
Therefore,
\begin{align*}
    \lambda(t, s)\ |\  N(\bar S) =1  \overset{d}{=}  \lambda'(t, s)\ |\  N'(\bar S) = 1
\end{align*}
implies that 
\begin{align*}
     \mu(t,s) + g(t-T_1,s - S_1)  \overset{d}{=}  \mu'(t,s) + g'(t-T_1',s - S_1'),
\end{align*}
and also that 
\begin{align*}
    g(t-T_1,s - S_1)  \overset{d}{=}  g'(t-T_1',s - S_1'),
\end{align*}
since $\mu(t,s) < +\infty$ under Assumption \ref{ass:mu}. Since $(T_1,S_1) \in \bar S$ and the equality above holds for any $1 \geq t \geq T_1$ and $s \in S$, it also holds for any $u = t-T_1 \in [0,a]$ and $v=s-S_1$ such that $\|v\|_\infty \leq b$ that
\begin{align*}
    g(u,v) = g'(u,v),
\end{align*}
which is equivalent to $g = g'$.

\subsection{Proof of Proposition \ref{prop:conc-ds}}

Before proving Proposition \ref{prop:conc-ds}, we state three technical lemmas. This first lemma defines a high probability event $\Omega_n$ on which the average and the maximum number of points on  $S = [0,1]^{d+1}$ are bounded. The second one provides upper bounds on the Kullback-Leibler (KL) divergence and on the deviations of the log-likelihood ratio of $f$ vs $f_0$, for $f$ sufficiently close to $f_0$, specifically $f \in B_\infty(\epsilon_n)$. The last lemma establishes the existence of tests with exponentially decaying Type-I and Type-II errors.

\begin{lemma}[High probability event]\label{lem:event}
For any $\alpha > 0$, there exists $\delta_0, c_\alpha > 0$ such that
\begin{align*}
    \Omega_n := \left \{ \frac{\underline{\mu}}{1 - \|g_0\|_1}  - \delta_n \leq \frac{1}{n}\sum_i N^i[0,1]^{d+1}  \leq \frac{\bar \mu}{1 - \|g_0\|_1} + \delta_n \right\} \cap \left \{ \sup_{i=1,\dots, n} N^i[0,1]^{d+1} \leq c_\alpha \log n \right \},
\end{align*}
with $\delta_n = \delta_0\frac{\log n }{\sqrt{n}}$ and $c > 0$. Under Assumption \ref{ass:boundedness},
\begin{align*}
    \mathbb P_0[\Omega_n] \geq 1 - 3 n^{-\alpha}.
\end{align*}
    
\end{lemma}

\begin{lemma}[Kullback-Leibler]\label{lem:KL}
    Under Assumption \ref{ass:boundedness} and if $\epsilon_n = o((\log n)^{-2})$, there exist $b_1,b_2 > 0$ such that for any $f \in B_\infty(\epsilon_n)$,
    \begin{align}
        &KL(f,f_0) := \mathbb E_0[\log L(N | f_0) - \log L(N|f)] \leq \kappa n \epsilon_n^2 ( 1 + o(1)) \label{eq:bound-kl}\\
        &\mathbb P_0( \log L(N | f_0) - \log L(N|f) > b_1 n \epsilon_n^2) \leq \frac{b_2}{n \epsilon_n^2},
    \end{align}
with
\begin{align}
    \kappa := \frac{4 \log 2}{\underline{\mu}}\left \{ 2 + 4\left(\frac{\bar \mu}{1 - \|g_0\|_1} + \Lambda_{0,2} \right) \right \}.
\end{align}
\end{lemma}

\begin{lemma}[Tests]\label{lem:tests}
Under Assumptions \ref{ass:sieves} and \ref{ass:boundedness}, there exists a test function  $\phi := \phi(N, \epsilon_n)$ such that
\begin{align*}
    &\mathbb E_0[\phi \mathds{1}_{\Omega_n}] = o(1) \\
    &\sup_{f \in \mathcal{F}_n}  \mathbb E_f[(1 - \phi)  \mathds 1_{\Omega_n} \mathds 1_{f \in A_n}] \leq e^{-b_2n \epsilon_n^2}
\end{align*}
where $b_2 > c_1$, $A_n$ is defined in \eqref{eq:def-an} and $\Omega_n$ is defined in Lemma \ref{lem:event}.
\end{lemma}

Proofs of these technical results are reported in Appendix \ref{app:proof-lemmas}. We now prove the proposition. For $M> 0$ and $\epsilon_n> 0$, define the subset of interest
    \begin{align}\label{eq:def-an}
        A_n = \{ f \in \mathcal{F} : d_S(f,f_0) > M \epsilon_n \}.
    \end{align}
    Note that since $\Pi(A_n|N_n) \in [0,1]$, \eqref{eq:conc-sd} is equivalent to $\mathbb E_0[\Pi( A_n | N_n)] = o(1)$, i.e., convergence in expectation. Given a test function $\phi := \phi(N, \epsilon_n) \in \{0,1\}$ and a high-probability event $\Omega_n$,
    \begin{align}
        \mathbb E_0[\Pi( A_n | N_n)] &=\mathbb E_0[\Pi( A_n | N_n)(\mathds{1}_{\Omega_n} + \mathds{1}_{\Omega_n^c})] \nonumber \\
        &\leq \mathbb E_0[\Pi( A_n | N_n)(\phi + 1-\phi)\mathds{1}_{\Omega_n}] + \mathbb E_0[\mathds{1}_{\Omega_n^c}] \nonumber \\
        &\leq \mathbb E_0[\phi \mathds{1}_{\Omega_n}] + \mathbb E_0[(1 - \phi) \Pi( A_n | N_n) \mathds{1}_{\Omega_n}] + \mathbb P_0[\Omega_n^c], \label{eq:decomp-post}
    \end{align}
    using that $\Pi( A_n | N_n) \leq 1$. With $\Omega_n$ as defined in Lemma \ref{lem:event}, we have $\mathbb P_0[\Omega_n^c] = o(1)$.  Moreover, we can write $\Pi( A_n | N_n)$ as
    \begin{align}\label{eq:decomp-an}
        \Pi( A_n | N_n) = \frac{\int_{A_n} L(N|f)\Pi(f) df}{\int_{\mathcal{F}} L(N|f) \Pi(f) df}  = \frac{\int_{\mathcal{F}} \mathds{1}_{f \in A_n}\frac{L(N|f)}{L(N|f_0)} \Pi(f) df}{\int_{\mathcal{F}} \frac{L(N|f)}{L(N|f_0)} \Pi(f) df}, 
    \end{align}
    defining 
    \begin{align*}
        D_n := \int_{\mathcal{F}} \frac{L(N|f)}{L(N|f_0)} \Pi(f) df 
    \end{align*}
    Note that for any deterministic sequence $\eta_n > 0$, 
    \begin{align*}
        \mathbb E_0[(1 - \phi) \Pi( A_n | N_n) \mathds{1}_{\Omega_n}] &= \mathbb E_0[(1 - \phi) \Pi( A_n | N_n) \mathds{1}_{\Omega_n}(\mathds{1}_{D_n \geq \eta_n} + \mathds{1}_{D_n < \eta_n})] \\
        &\leq \mathbb P_0[D_n < \eta_n] + \mathbb E_0[(1 - \phi) \Pi( A_n | N_n) \mathds{1}_{\Omega_n}\mathds{1}_{D_n > \eta_n}].
    \end{align*}
    Then, using \eqref{eq:decomp-post}, \eqref{eq:decomp-an} and the inequality above,
    \begin{align*}
        \mathbb E_0[\Pi( A_n | N_n)] &\leq \mathbb E_0[\phi \mathds{1}_{\Omega_n}] +  \mathbb P_0[D_n < \eta_n] \\
        &
        \qquad + \mathbb E_0\left[\frac{1}{D_n}(1 - \phi) \int_{\mathcal{F}}   \mathds{1}_{f \in A_n} \frac{L(N|f)}{L(N|f_0)} \Pi(f)  df   \mathds{1}_{D_n \geq \eta_n} \mathds{1}_{\Omega_n}\right]   + \mathbb P_0[\Omega_n^c] \\
        &\leq  \mathbb E_0[\phi \mathds{1}_{\Omega_n}] +  \mathbb P_0[D_n < \eta_n] + \frac{1}{\eta_n} \mathbb E_0\left[(1 - \phi) \int_{\mathcal{F}_n} \frac{L(N|f)}{L(N|f_0)}   \mathds{1}_{f \in A_n} \Pi(f) df  \mathds{1}_{\Omega_n}\right] \\
        &\qquad + \mathbb P_0[\Omega_n^c] + \frac{1}{\eta_n} \mathbb{E}_0\left[ \int_{\mathcal{F}_n^c} \frac{L(N|f)}{L(N|f_0)}   \Pi(f) df \right].
    \end{align*}
    using in the second inequality that $\mathcal{F} = \mathcal{F}_n \cup \mathcal{F}_n^c$. By Fubini's theorem,
    \begin{align*}
       \mathbb{E}_0\left[ \int_{\mathcal{F}_n^c} \frac{L(N|f)}{L(N|f_0)}   \Pi(f) df \right] 
        &=   \int_{\mathcal{F}_n^c} \mathbb{E}_0\left[\frac{L(N|f)}{L(N|f_0)}\right] \Pi(f) df =  \int_{\mathcal{F}_n^c} \Pi(f) df = \Pi(\mathcal{F}_n^c)
    \end{align*}
    using that $\mathbb{E}_0[\frac{L(N|f)}{L(N|f_0)}] = 1$. Defining $\eta_n = \Pi(B_\infty(\bar \epsilon_n) ) e^{- b_1 n \bar \epsilon_n^2} \geq e^{- (b_1 + c_1) n \bar \epsilon_n^2}$ under Assumption \ref{ass:prior-mass} for some $b_1 > 0$ such that $b_1 < c_2 - c_1$, then, using that under Assumption \ref{ass:sieves}, $\Pi(\mathcal{F}_n^c) \leq e^{- c_2 n \bar \epsilon_n^2}$, we have
    \begin{align*}
        \frac{1}{\eta_n}  \Pi(\mathcal{F}_n^c) \leq e^{(b_1 + c_1 - c_2) n \bar \epsilon_n^2} = o(1).
    \end{align*}
    Using the same computations as in \cite{donnet2020nonparametric} (proof of Theorem 1),
    \begin{align*}
        \mathbb P_0[&D_n < \Pi(B_\infty(\bar \epsilon_n) ) e^{- b_1 n \bar \epsilon_n^2}] \\
        &\qquad \leq \frac{1}{\Pi(B_\infty(\bar \epsilon_n) ) (1 - e^{- b_1 n \bar \epsilon_n^2})} \int_{B_\infty(\bar \epsilon_n)} \mathbb P_0( \log L(f_0|N) - \log L(f|N) > b_1 n \bar \epsilon_n^2) \Pi(f) df.
    \end{align*}
    Using Lemma \ref{lem:KL}, if $b_1 > \kappa$, we have for any $f \in B_\infty(\bar \epsilon_n)$,
    \begin{align*}
        \mathbb P_0( \log L(f_0|N) - \log L(f|N) > b_1 n \bar \epsilon_n^2) \leq \frac{b_2}{n \bar \epsilon_n^2},
  \end{align*}
    with $b_2 = \frac{2\kappa}{(b_1 - \kappa)^2}$ which implies that
    \begin{align*}
        \mathbb P_0[D_n < \Pi(B_\infty(\epsilon_n) ) e^{- b_1 n \epsilon_n^2}]  \leq \frac{b_2}{n \bar \epsilon_n^2 (1 - e^{-b_1n \bar \epsilon_n^2})} = o(1),
    \end{align*}
    under the assumption that $n \bar \epsilon_n^2 \to \infty$. Note that since $c_2 > c_1 + \kappa$ under Assumption \ref{ass:sieves}, then there exists $b_1 \in (\kappa, c_2-c_1)$ (e.g., $\kappa + \frac{c_2-c_1-\kappa}{2}$).

    Moreover, using Lemma \ref{lem:tests}, we can find $\phi$ such that
    \begin{align*}
        &\mathbb E_0[\phi \mathds{1}_{\Omega_n}] = o(1) \\
        &\sup_{f \in \mathcal{F}_n} \mathbb E_f[(1 - \phi) \mathds{1}_{\Omega_n} \mathds{1}_{f \in A_n}] \leq e^{-b_2 n \epsilon_n^2} = o( \Pi( B_\infty(\epsilon_n)) e^{-b_1 n \bar \epsilon_n^2}) = o(\eta_n).
    \end{align*}
    Thus, using again Fubini's theorem,
    \begin{align*}
       \frac{1}{\eta_n} \mathbb E_0\left[(1 - \phi) \int_{\mathcal{F}_n} \frac{L(N|f)}{L(N|f_0)}   \mathds{1}_{f \in A_n} \Pi(f) df  \mathds{1}_{\Omega_n}\right] &= \frac{1}{\eta_n}  \mathbb \int_{\mathcal{F}_n} \mathbb E_0\left[\frac{L(N|f)}{L(N|f_0)} (1 - \phi)  \mathds{1}_{f \in A_n} \mathds{1}_{\Omega_n}\right] \Pi(f) df   \\
        &=  \frac{1}{\eta_n}  \int_{\mathcal{F}_n} \mathbb E_f[(1 - \phi)    \mathds{1}_{\Omega_n}\mathds{1}_{f \in A_n}]  \Pi(f) df   \\
        &\leq \frac{\Pi(\mathcal{F}_n)}{\eta_n}e^{-b_2 n \bar \epsilon_n^2}   \leq \frac{ e^{-b_2 n \bar \epsilon_n^2}}{\eta_n} = o(1),
    \end{align*}
    and this concludes this proof.

\subsection{Proof of Proposition \ref{prop:conc-l1}}

We first state a lemma that provides a bound on the expectation under any $f \in A_{n}^c$ of the random variables $(Z_i)$ (as defined in \eqref{eq:zi}). Its proof can be found in Appendix \ref{app:proof-lemmas}.

\begin{lemma}\label{lem:ef}
For any $f \in A_{n}^c$, there exists a constant $p_0 >0$ such that on $\Omega_n$,
    \begin{align*}
        \mathbb E_f[Z_1] \geq p_0 \|f-f_0\|_1.
    \end{align*}
\end{lemma}

For $M' > 0$,  let 
    \begin{align*}
        A_{n,1} = \{ f \in \mathcal{F} : \|f- f_0\|_1 > M' \epsilon_n \}.
    \end{align*}
    Using that $A_{n,1} =  (A_n \cap A_{n,1}) \cup  (A_n^c \cap A_{n,1})$, we have
    \begin{align*}
         \mathbb E_0[\Pi( A_{n,1} | N_n)] \leq  \mathbb E_0[\Pi( A_n | N_n)] + \mathbb E_0[\Pi( A_n^c \cap A_{n,1}  | N_n)] = o(1) + \mathbb E_0[\Pi( A_n^c \cap A_{n,1}  | N_n)].
    \end{align*}
    using Proposition \ref{prop:conc-ds}. Applying the same decomposition as in the proof of Proposition \ref{prop:conc-ds}, we have
        \begin{align*}
        \mathbb E_0[\Pi(A_n^c \cap A_{n,1}  | N_n)] 
        &\leq  \mathbb P_0[D_n < \eta_n] + \frac{1}{\eta_n} \int_{\mathcal{F}_n} \mathbb E_f[\mathds{1}_{\Omega_n} \mathds{1}_{f \in A_{n,1} \cap A_n^c}] \Pi(f) df   \\
        &+ \mathbb P_0[\Omega_n^c] + \frac{1}{\eta_n} \Pi(\mathcal{F}_n^c) \\
        &\leq  o(1) + \frac{1}{\eta_n} \sup_{f \in \mathcal{F}_n\cap  A_{n,1}} \mathbb E_f[\mathds{1}_{\Omega_n} \mathds{1}_{f \in A_n^c}].
    \end{align*}
    Recall that
    \begin{align*}
        f \in A_n^c &\iff d_{S}(f,f_0) \leq M \epsilon_n  \\
        &\iff \frac{1}{n} \sum_{i=1}^n \int |\lambda^i_{t,s}(f) -  \lambda^i_{t,s}(f_0)| dt ds \leq M\epsilon_n
    \end{align*}
    where the integral above is over the whole observation domain $[0,1] \times [0,1]^{d+1}$. Recall that for any sequence $N^i$, we call $t^i_1, t^i_2, \dots$ the times of the events. Defining
    \begin{align}\label{eq:zi}
        Z_i := \int_0^{t^i_2} \int_{[0,1]^{d}} |\lambda^i_{t,s}(f) -  \lambda^i_{t,s}(f_0)| dt ds.
    \end{align}
    By convention, if $t^i_2$ does not exist (i.e., the sequence has only one or no event at all), we set $t^i_2 = 1$. Note that the $(Z_i)_i$ are i.i.d. and that
    \begin{align*}
       \mathbb P_f \left[\{f \in A_n^c\} \cap \{\Omega_n\} \right] \leq  \mathbb P_f \left[\frac{1}{n} \sum_{i=1}^n Z_i \leq M\epsilon_n \cap \{\Omega_n\}\right] 
    \end{align*}
    
    For any $f \in A_{n}^c$, using Lemma \ref{lem:ef}, there exists a constant $p_0 >0$ such that on $\Omega_n$,
    \begin{align*}
        \mathbb E_f[Z_i] \geq p_0 \|f-f_0\|_1.
    \end{align*}
    Thus,
    \begin{align*}
        \mathbb P_f \left[\frac{1}{n} \sum_{i=1}^n Z_i \leq M\epsilon_n \cap \{\Omega_n\}\right] &= \mathbb P_f \left[\frac{1}{n} \sum_{i=1}^n ( Z_i -\mathbb E_f[Z_i] ) \leq M\epsilon_n - \frac{1}{n}\sum_i \mathbb E_f[Z_i] \cap \{\Omega_n\}\right]  \\
        &\leq \mathbb P_f \left[\frac{1}{n} \sum_{i=1}^n (Z_i -\mathbb E_f[Z_i] ) \leq M\epsilon_n - p_0 \|f-f_0\|_1 \right] \\ 
        &\leq \mathbb P_f \left[\frac{1}{n} \sum_{i=1}^n (Z_i -\mathbb E_f[Z_i] ) \leq  - p_0 \|f-f_0\|_1 /2 \right],
    \end{align*}
    using that $f \in A_{n,1}$ and for $M' > 2M/p_0$.
    Moreover,
    \begin{align*}
        Z_i &=\int_0^{t^i_2} \int_{[0,1]^{d}} |\lambda^i_{t,s}(f) -  \lambda^i_{t,s}(f_0)|dt ds\\
        &= \int_0^{t^i_1} \int_{[0,1]^{d}} |\lambda^i_{t,s}(f) -  \lambda^i_{t,s}(f_0)| dt ds + \int_{t^i_1}^{t^i_2} \int_{[0,1]^{d}} |\lambda^i_{t,s}(f) -  \lambda^i_{t,s}(f_0)| dt ds \\
        &= \int_0^{t^i_1}\int_{[0,1]^{d}} |\mu(t,s) -  \mu_0(t,s)| dt ds \\
        &+ \int_{t^i_1}^{t^i_2}\int_{[0,1]^{d}} |\mu(t,s) + g(t-t_1^i, s-s_1^i) -  \mu_0(t,s) - g_0(t-t_1^i,  s-s_1^i)| dt ds \\
        &\leq \int_0^{t^i_2}\int_{[0,1]^{d}} |\mu(t,s) -  \mu_0(t,s)| dt ds \\
        &+ \int_{t^i_1}^{t^i_2}\int_{[0,1]^{d}} |g(t-t_1^i,  s-s_1^i) - g_0(t-t_1^i,  s-s_1^i)| dt ds \\
        &\leq \|\mu - \mu_0\|_1 +  \int_0^{t^i_2 - t^i_1}\int_{[0,1]^{d}} |g(u,  s-s_1^i) - g_0(u,  s-s_1^i)| du ds \\
        &\leq \|\mu - \mu_0\|_1 + \|g - g_0\|_1 = \|f - f_0\|_1.
    \end{align*}
    Thus,
    \begin{align*}
        &\mathbb E_f[Z_i^k] \leq \|f - f_0\|^k, \qquad k \geq 2.
    \end{align*}
    Applying Bernstein's inequality, we obtain
    \begin{align*}
         \mathbb P_f \left[\frac{1}{n} \sum_{i=1}^n Z_i - \mathbb E_f[Z_i] \leq  -p_0 \|f-f_0\|_1  /2 \right]  &\leq  \exp \left( - \frac{p_0^2 n\|f-f_0\|_1 ^2}{8\|f-f_0\|_1^2(1 + p_0/3) }\right)  \\
         &
         =\exp \left( - \frac{p_0^2 n}{8(1 + p_0/3) }\right) \\
         &= o(e^{-cn\epsilon_n^2}),
    \end{align*}
    for any constant $c > 0$. Therefore we can conclude that for any $f \in \mathcal{F}_n \cap A_{n,1}$,
    \begin{align*}
        \mathbb E_f[\mathds{1}_{\Omega_n} \mathds{1}_{f \in A_n^c}] = o(e^{-(b_1 + c_1) n \epsilon_n^2}) = o(\eta_n),
    \end{align*}
    recalling that $\eta_n = \Pi(B_\infty(\bar \epsilon_n) ) e^{- b_1 n \bar \epsilon_n^2} \geq e^{- (b_1 + c_1) n \bar \epsilon_n^2}$ with $\bar \epsilon_n \leq \epsilon_n$. Thus,
    \begin{align*}
        \frac{1}{\eta_n} \sup_{f \in \mathcal{F}_n\cap  A_{n,1}} \mathbb E_f[\mathds{1}_{\Omega_n} \mathds{1}_{f \in A_n^c}] = o(1),
    \end{align*}
    and this concludes the proof of this proposition.

\subsection{Proof of Proposition \ref{prop:GP1}}

We first show that Assumption \ref{ass:prior-mass} holds. Note that
\begin{align*}
    B_\infty(\epsilon_n) \supset \{ \mu : \|\mu - \mu_0\|_\infty \leq \epsilon_n/2\} \times \{ g : \|g - g_0\|_\infty \leq \epsilon_n/2\},
\end{align*}
therefore
\begin{align*}
    \Pi(B_\infty(\epsilon_n) ) \geq \Pi_\mu( \|\mu - \mu_0\|_\infty \leq \epsilon_n/2 ) \Pi_g(  \|g - g_0\|_\infty \leq \epsilon_n/2).
\end{align*}
Let us consider the first term $\Pi_\mu( \|\mu - \mu_0\|_\infty \leq \epsilon_n/2 )$. Since $\mu_0(x) \in [\underline{\mu}, \bar \mu]$ and $\sigma$ is bijective from $\mathbb R$ to $(0,C)$, let $\sigma^{-1}$ the inverse of $\sigma$ on $(0,C)$ and $\nu_0 := \sigma^{-1}(\mu_0)$. Furthermore, $\mu_0 \in C^{\alpha}(S)$ and $\sigma$ infinitely smooth imply that $\nu_0 \in C^{\alpha}(S)$. Since $\sigma$ is $L$-Lipschitz,
\begin{align*}
    \Pi_\mu( \|\mu - \mu_0\|_\infty \leq \epsilon_n/2 ) \geq \Pi_\nu( \|\nu - \nu_0\|_\infty \leq \epsilon_n/(2L) ),
\end{align*}
where $\Pi_\nu$ is the GP prior with kernel $k_\nu$. 


Recall that the RKHS of the Mat\'ern Kernel with parameter $\tau \leq \alpha$, denoted by $\mathcal{H}_{MAT}^\tau$, is the Sobolev space with smoothness $\tau + \frac{d+1}{2}$. Since $\alpha \geq \tau$, $w_0 \in \mathcal{H}_{MAT}^\tau$. Using Lemma B.1 in \cite{GiordanoKirichenkoRousseau2025} there exists $L_1 > 0$ such that
\begin{align*}
    \Pi_\nu( \|\nu - \nu_0\|_\infty \leq \epsilon_n/(2L) ) \geq e^{-L_1 n \epsilon_n^2},
\end{align*}
for any $\epsilon_n \to 0$ such that $\epsilon_n \gtrsim n^{-\tau/(2\tau + d + 1)}$.
Consider now the second term $\Pi_g( g : \|g - g_0\|_\infty \leq \epsilon_n/2)$. Let $g_{0,n} := \max(g_0, \epsilon_n/4)$ so that $\|g_0 - g_{0,n}\|_\infty \leq \epsilon_n/4$. Therefore,
\begin{align*}
    \Pi_g( g : \|g - g_0\|_\infty \leq \epsilon_n/2) \geq \Pi_g( g : \|g - g_{0,n}\|_\infty \leq \epsilon_n/4). 
\end{align*}
Defining $\phi_{0,n} = \sigma^{-1}(g_{0,n})$, using the same argument as for $\nu_0$ and since $\sigma$ is $L$-Lipschitz, there exists $L_2 > 0$ such that 
\begin{align*}
    \Pi_g(  \|g - g_{0,n}\|_\infty \leq \epsilon_n/4) \geq \Pi_\phi(  \|\phi - \phi_{0,n}\|_\infty \leq \epsilon_n/(4L)) \geq e^{-L_2 n \epsilon_n^2}, 
\end{align*}
with $\Pi_\phi$ the GP prior with Mat\'ern kernel on $[0,a] \times [0,b]^d$ and $\epsilon_n$ as before.
Therefore,
\begin{align*}
    \Pi(B_\infty(\epsilon_n) ) &\geq  \Pi_\nu( \|\nu - \nu_0\|_\infty \leq \epsilon_n/(2L) ) \Pi_\phi(  \|\phi - \phi_{0,n}\|_\infty \leq \epsilon_n/(4L))\\
    &\geq  \Pi_\nu( \|\nu - \nu_0\|_\infty \leq \epsilon_n/(4L) ) \Pi_\phi(  \|\phi - \phi_{0,n}\|_\infty \leq \epsilon_n/(4L)) \\
    &\geq e^{-(L_1 + L_2)n\epsilon_n^2}, 
\end{align*}
which demonstrates Assumption \ref{ass:prior-mass} with $c_1 := L_1 + L_2$ and $\bar \epsilon_n = \epsilon_n \asymp n^{-\tau/(2\tau + d + 1)}$.

We now prove that Assumption \ref{ass:sieves} holds. For $M_1, M_2 > 0$,  let
\begin{align*}
    \mathcal{M}_n = M_1 \epsilon_n \mathbb B_1 + M_2 \sqrt{n} \epsilon_n \mathbb H_1,
\end{align*}
with $\mathbb B_1, \mathbb H_1$ are  the unit balls respectively in $L_\infty(S)$ and in $\mathcal{H}_{MAT}^\tau$ (w.r.t. the corresponding Sobolev norm). Using Lemma B.2 in \cite{GiordanoKirichenkoRousseau2025}, for any $\epsilon_n \to 0$ such that $\epsilon_n \gtrsim n^{-\tau/(2\tau + d + 1)}$ and $R_1>0$, there exists $M_1, M_2 > 0$ and $R_2>0$ such that
\begin{align}
    &\Pi_\nu(\mathcal{M}_n^c) \leq e^{-R_1 n\epsilon_n^2} \nonumber \\
    &\log C(\epsilon_n, \mathcal{M}_n, \|\cdot \|_\infty) \leq R_2 n\epsilon_n^2. \label{eq:bound-covering}
\end{align}
Similarly we construct
\begin{align*}
 \mathcal{N}_n = M_1 \epsilon_n  \mathbb{\bar B}_1 + M_2 \sqrt{n} \epsilon_n  \mathbb{\bar H}_1,
\end{align*}
with $ \mathbb{\bar B}_1,  \mathbb{\bar H}_1$  the unit balls respectively in $L_\infty([0,a] \times [0,b]^d)$ and to the analog of $\mathcal{H}_{MAT}^\tau$ on $[0,a] \times [0,b]^d$ (w.r.t. the corresponding Sobolev norm) and we have
\begin{align}
    &\Pi_\phi(\mathcal{N}_n^c) \leq e^{-R_1 n\epsilon_n^2} \\
    &\log C(\epsilon_n, \mathcal{N}_n, \|\cdot \|_\infty) \leq R_2 n\epsilon_n^2.\label{eq:bound-covering2}
\end{align}

We then construct the sieves as
\begin{align*}
    \mathcal{F}_n = \sigma(\mathcal{M}_n) \times \sigma (\mathcal{N}_n).
\end{align*}
We obtain
\begin{align*}
    \Pi(\mathcal{F}_n) = \Pi_\mu(\sigma(\mathcal{M}_n))  \Pi_g(\sigma(\mathcal{N}_n))  &\geq \Pi_\nu(\mathcal{M}_n) \Pi_\phi(\mathcal{N}_n) \\
    &\geq (1 - e^{-R_1n\epsilon_n^2})^2 \\
    &\geq 1 - 2 e^{-R_1 n\epsilon_n^2} \geq 1 - e^{- \frac{R_1}{2} n\epsilon_n^2}.
\end{align*}
Therefore, choosing $R_1 > 2(c_1 + \kappa)$ we obtain the first part of Assumption \ref{ass:sieves}. Moreover, since $\sigma$ is $L$-Lipschitz,
\begin{align*}
    C(\epsilon, \mathcal{F}_n, \|\cdot \|_\infty) &= C(\epsilon, \sigma(\mathcal{M}_n), \|\cdot \|_\infty)C(\epsilon, \sigma(\mathcal{M}_n), \|\cdot \|_\infty) \\
    &\leq C(L \epsilon, \mathcal{M}_n, \|\cdot \|_\infty)C( L \epsilon, \mathcal{M}_n, \|\cdot \|_\infty).
\end{align*}
Therefore, using \eqref{eq:bound-covering} and \eqref{eq:bound-covering2}, we obtain
\begin{align*}
    C(\epsilon_n, \mathcal{F}_n, \|\cdot \|_\infty) \leq e^{2R_2 L^2 n \epsilon_n^2},
\end{align*} 
which proves the second part of  Assumption \ref{ass:sieves} with $c_3:= 2R_2 L^2$ and $\epsilon_n \asymp n^{-\tau/(2\tau + d + 1)}$.


\subsection{Proof of Proposition \ref{prop:GP2}}

To verify Assumptions \ref{ass:prior-mass} and \ref{ass:sieves} for the squared-exponential kernel with Inverse-Gamma prior on the length-scale, we follow the same steps as in the proof of Proposition \ref{prop:GP1} and  apply Theorem 3.1 from \cite{van2009adaptive}.


In fact, if $\epsilon_n = 2L K (\log n )^{\frac{d+1}{2 + (d+1)/\alpha}} n^{-\frac{\alpha}{2\alpha+d+1}}$ with $K>0$ sufficiently large, then
\begin{align*}
    \Pi_\nu( \|\nu - \nu_0\|_\infty \leq \epsilon_n/(2L) ) \geq e^{-\frac{n}{4L^2}\epsilon_n^2}.
\end{align*}
Using the same arguments as before, we obtain that Assumption \ref{ass:prior-mass} holds with  $c_1 := L_1 + L_2$ and $\bar \epsilon_n \asymp (\log n)^{\frac{d+1}{2 + (d+1)\alpha}} n^{-\tau/(2\tau + d + 1)}$. Moreover there exist $\mathcal{M}_n := \{\mu:S \to \mathbb R_+\}$, $\mathcal{G}_n : [0,a]\times [0,b]^d \to \mathbb R_+$, $M_1, M_2>0$ such that
\begin{align*}
    &\Pi(\mathcal{M}_n^c) \leq e^{-\frac{n}{L} \bar \epsilon_n^2} \\
    &\Pi(\mathcal{G}_n^c) \leq e^{-\frac{n}{L} \bar \epsilon_n^2} \\
    &C(M_1  \epsilon_n, \mathcal{M}_n,\|\cdot\|_\infty) \leq e^{ M_1^2 n \epsilon_n^2}\\
    &C(M_2  \epsilon_n, \mathcal{G}_n,\|\cdot\|_\infty) \leq e^{ M_2^2 n \epsilon_n^2},
\end{align*}
with $ \epsilon_n \asymp (\log n)^{\frac{d + 1}{2}} \bar \epsilon_n $, which allows to verify Assumption \ref{ass:sieves}.

\section{Conclusion}

The results of this work advance the theoretical understanding of Bayesian nonparametric inference for point processes by establishing posterior contraction rates for multivariate, non-stationary spatio-temporal Hawkes processes. By extending existing analyses for temporal Hawkes processes to the general spatio-temporal setting, we provide the first theoretical guarantees for Bayesian nonparametric learning of self-exciting mechanisms where both background and triggering components evolve across time and space. The framework is broadly relevant to applications involving complex dependency structures, such as seismic activity, neural spike trains, social or financial interactions, and information diffusion on networks, where events exhibit both temporal and spatial excitation.

Under mild regularity conditions on the true intensity and the prior, we show that flexible priors, particularly hierarchical Gaussian processes with squared-exponential  kernel yield asymptotically optimal concentration rates. Future work could extend these results to multivariate or nonlinear Hawkes processes, where interactions among multiple latent components or nonlinear excitation effects introduce new theoretical and computational challenges. Additionally, the inclusion of covariates would be an interesting direction as it would enhance the spatial learning.

\textbf{Acknowledgements}
For this project XM has received funding from the European Union's Horizon Europe research and innovation programme under the Marie Skłodowska-Curie grant agreement 101151781.

\newpage

\appendix

\section{Technical lemmas}

\subsection{Bounds on stochastic distance}

The next two lemmas provide lower and upper bounds on the stochastic distance using the $L_1$-norm, on the high probability event $\Omega_n$ (defined in Lemma \ref{lem:event}).

\begin{lemma}\label{lem:bound-ds-by-f1}
    On $\Omega_n$, for any $f,f'$,
    \begin{align*}
        d_{S}(f,f') \leq N_0 \|f - f'\|_1.
    \end{align*}
    with $N_0 = \bar \mu + \|g_0\|_1  + 1$.
\end{lemma}

\begin{proof}
    \begin{align*}
        d_{S}(f,f') &= \frac{1}{n} \sum_{i} \int |\lambda^i_f(t,s) - \lambda^i_{f'}(t,s)|dt ds \\
        &\leq \frac{1}{n}\sum_{i} \int |\mu(t,s) - \mu'(t,s)|dt ds + \frac{1}{n} \sum_{i} \int \left|\sum_{t_j<t} g(t-t_j,s-s_j) - g'(t-t_j,s-s_j) \right|dt ds \\
        &\leq  \|\mu(t,s) - \mu'(t,s)\|_1 + \frac{1}{n} \sum_{i}  \sum_{t_j<t} \int \left|g(t-t_j,s-s_j) - g'(t-t_j,s-s_j) \right|dt ds \\
        &\leq  \|\mu(t,s) - \mu'(t,s)\|_1 + \frac{1}{n} \sum_{i}  \sum_{t_j<t} \|g - g'\|_1 \\
        &\leq \|\mu - \mu'\|_1 + \|g - g'\|_1 \frac{1}{n}\sum_{i} N^i[0,1]^{d+1} \\
        &\leq N_0( \|\mu - \mu'\|_1 +  \|g - g'\|_1) = N_0 \|f-f'\|_1,
    \end{align*}
    with $N_0 = \bar \mu + \|g_0\|_1 + 1$ on $\Omega_n$.
\end{proof}

\begin{lemma}\label{lem:sd-bound}
    For any $f \in \mathcal{F}$, on $\Omega_n$,
    \begin{align}\label{eq:ineq-ds}
    &- d_S(f,f_0) +  \|\mu_0\|_1 + \|g_0\|_1 \frac{e_0}{2} \leq \|\mu\|_1 + \|g\|_1 \frac{\sum_{i=1}^n N^i(S) }{n} \leq  \|\mu_0\|_1 + \frac{3 e_0}{2} \|g_0\|_1  +   d_S(f,f_0) 
\end{align}
\end{lemma}

\begin{proof}
On one hand,
\begin{align}
    d_S(f,f_0) &\leq \int_{[0,1]} \int_{[0,1]^d} | \mu(t,x) -   \mu_0(t,x)| dt dx + \frac{\sum_{i=1}^n N^i(S) }{n} \int_{[0,a]} \int_{[0,b]} | g(t,x) -   g_0(t,x)| dt dx \nonumber \\
    &\leq \| \mu -   \mu_0\|_1 + \frac{\sum_{i=1}^n N^i(S) }{n} \| g -   g_0\|_1 \label{eq:d_S_bound}
\end{align}
with

\begin{align*}
     &\| \mu -   \mu_0\|_1 = \int_{[0,1]} \int_{[0,1]^d} | \mu(t,x) -   \mu_0(t,x)| dt dx \\
     & \| g -   g_0\|_1 = \int_{[0,a]} \int_{[0,b]} | g(t,x) -   g_0(t,x)| dt dx.
\end{align*}
On the other hand,
\begin{align*}
    d_S(f,f_0) &\geq \frac{1}{n} \left| \sum_i \int_{[0,1]} \int_{[0,1]^d} (\lambda^i_{t,x}(f) -  \lambda^i_{t,x}(f_0)) dt dx \right| \\
    &\geq \frac{1}{n} \left| \sum_i (\|\mu\|_1 + N^i(S)\|g\|_1 - \|\mu_0\|_1 - N^i(S)\|g_0\|_1 \right|
\end{align*}
from which we deduce that
\begin{align*}    
    &- d_S(f,f_0) +  \|\mu_0\|_1 + \|g_0\|_1 \frac{\sum_{i=1}^n N^i(S) }{n} \leq \|\mu\|_1 + \|g\|_1 \frac{\sum_{i=1}^n N^i(S) }{n} \leq  \|\mu_0\|_1 + \|g_0\|_1 \frac{\sum_{i=1}^n N^i(S) }{n}  +  d_S(f,f_0)
\end{align*}
The previous inequality basically implies that if $d_S(f,f_0)$ is small, then also the $L_1$-norm of $(\mu, g)$ is bounded by the $L_1$-norm of the true parameter, provided that $\frac{\sum_{i=1}^n N^i(S) }{n}$ is concentrated around its expectation. On $\Omega_n$,
\begin{align*}
    \frac{\underline{ \mu}}{1 - \|g_0\|_1} - \delta_n \leq \frac{1}{n}\sum_i N^i[0,1]^{d+1}  \leq  \frac{\bar \mu}{1 - \|g_0\|_1} + \delta_n,
\end{align*}
thus with $e_0 = \frac{\bar \mu}{1 - \|g_0\|_1}$ and $n$ large enough, we obtain \eqref{eq:ineq-ds}.

\end{proof}

\subsection{Bernstein inequalities}

The next two lemmas are two useful versions of Bernstein inequalities for point processes. 

\begin{lemma}\label{lem:bernstein2}
    Let $N = (N^i)_{i=1, \dots, n}$ $n$ i.i.d. spatio-temporal Hawkes point processes on $[0,1] \times [0,1]^{d}$ with parameter $f$ satisfying Assumption \ref{ass:boundedness}. Let $(S_i)_{i=1, \dots, n} \subset [0,1]^{d+1}$ possibly random independent subsets and $v > 0$ a deterministic constant such that
    \begin{align*}
      \sum_{i = 1}^n \Lambda^i(S_i) =  \sum_i  \int_{S_{1,i}} \lambda^i_{t,s}(f) dt ds \leq v,
    \end{align*}
    where $\Lambda^i(S_i) = \int_{S_i} \lambda_{t,s}^i(f) dt ds$. Then for any $x > 0$,
\begin{align}
    &\mathbb P\left( \sum_{i=1}^n N^i(S_i) - \Lambda^i(S_i) \geq \sqrt{2vx} + \frac{x}{3} \right) \leq e^{-x}, \label{eq:bernstein3} \\
    &\mathbb P\left(  \sum_{i=1}^n N^i(S_i) - \Lambda^i(S_i) \leq -  \sqrt{2vx} - \frac{x}{3}  \right) \leq e^{-x},\label{eq:bernstein4}
\end{align}
where $\Lambda^i(S_i) = \int_{S_i} \lambda_{t,s}^i(f) dt ds$.
\end{lemma}

\begin{proof}

We prove the bound \eqref{eq:bernstein3} on the right tail probability. The left tail probability \eqref{eq:bernstein4}  can be proven following the same strategy. This proof is structured in three main steps: 1) an exponential moment for a re-centered version of $\sum_i N^i(S_i)$ is established; 2) The Chernoff inequality is used to bound the tail probability; 3) a lower bound for any small enough value of   the free parameter is applied to obtain the result.


\paragraph{Step 1} Let  
\begin{align*}
    E := e^{\theta \sum_i (N^i(S_i) - \Lambda^i(S_i)) - \phi(\theta) \sum_i \Lambda^i(S_i)}
\end{align*}
with $\phi(x) = e^u -u -1$ and for any 
$\theta>0$. We will prove that $\mathbb E[E] \leq 1$ where we use the shortened notation $\mathbb E := \mathbb E_f$. Firstly, since the variables 
\begin{align*}
    \theta  (N^i(S_i) - \Lambda^i(S_i)) - \phi(\theta) \Lambda^i(S_i)
\end{align*}
are stochastically independent, then
\begin{align*}
    \mathbb E[E] = \prod_{i = 1}^n \mathbb E[ e^{\theta (N^i(S_i) - \Lambda^i(S_i)) - \phi(\theta) \Lambda^i(S_i)}].
\end{align*}
Moreover, since $\phi(\theta) \geq \frac{\theta^2}{2} + \frac{\theta^3}{6} = \frac{\theta^2}{2}  ( 1 + \frac{\theta}{3})$, then
\begin{align}
    \mathbb E[ e^{\theta (N^i(S_i) - \Lambda^i(S_i)) - \phi(\theta) \Lambda^i(S_i)}] &\leq \mathbb E[ e^{\theta (N^i(S_i) - \Lambda^i(S_i)) - \frac{\theta^2}{2}  ( 1 + \frac{\theta}{3}) \Lambda^i(S_i)}] \nonumber \\
    &= \mathbb E[ e^{\theta (N^i(S_i) - \Lambda^i(S_i) (1 +  \frac{\theta}{2}  ( 1 + \frac{\theta}{3})))}] \nonumber \\
    &\leq \mathbb E[ e^{\theta ( N^i(S_i) - \Lambda^i(S_i))}] \mathbb E[ e^{- \Lambda^i(S_i)  \frac{\theta^2}{2}  ( 1 + \frac{\theta}{3}))}] \label{eq:ineq-product}
\end{align}
using Cauchy-Schwarz inequality in the last inequality. Under Assumption \ref{ass:boundedness}, $N^i$ is non-explosive and we can easily prove that it admits exponential moments. To see this, first observe that
\begin{align*}
    N^i(S_i) - \Lambda^i(S_i) \leq N^i([0,1]^{d+1}) - \Lambda^i([0,1]^{d+1})
\end{align*}
and that $N^i([0,1]^{d+1})$ is stochastically bounded by $\bar N([0,1])$ where $\bar N$ is a temporal point process with constant background $\bar \mu$, temporal kernel $\bar g(t) = \int_s g(t,s)ds$ and compensator $\bar \Lambda$. Thus,
\begin{align*}
     \mathbb E[ e^{\theta ( N^i(S_i) - \Lambda^i(S_i))}] \leq \mathbb E[ e^{\theta (N^i([0,1]^{d+1}) - \Lambda^i([0,1]^{d+1})) }] \leq \mathbb E[ e^{\theta (\bar N([0,1]) - \bar \Lambda([0,1])) }].
\end{align*}
By Theorem 2 in \cite{bremaud81},
\begin{align*}
   \mathbb E[ e^{\theta (\bar N([0,1]) - \bar \Lambda([0,1])) }] \leq 1,
\end{align*}
which therefore implies that $ \mathbb E[ e^{\theta ( N^i(S_i) - \Lambda^i(S_i))}] \leq 1$. Moreover 
$\Lambda_i(S_i) \geq 0$ and $\theta > 0$, then
\begin{align*}
    \mathbb E[ e^{-\theta \Lambda^i(S_i) (1 +  \frac{\theta}{2}  ( 1 + \frac{\theta}{3}))}] \leq 1.
\end{align*}
In light of \eqref{eq:ineq-product}, we obtain 
\begin{align*}
    \mathbb E[ e^{\theta (N^i(S_i) - \Lambda^i(S_i)) - \phi(\theta) \Lambda^i(S_i)}] \leq 1,
\end{align*}
and thus that $ \mathbb E[E] \leq 1$ as we wished to prove.

\paragraph{Step 2} Using the standard Chernoff inequality, for any $x> 0$, we have
\begin{align*}
    &\mathbb P \left( \theta \sum_i (N^i(S_i) - \Lambda^i(S_i))  - \phi(\theta) \sum_i \Lambda^i(S_i) \geq x \right)
    = \mathbb P \left(  E \geq e^{x} \right) \leq \mathbb E[E] e^{-x}  \leq e^{-x}.
\end{align*}
Additionally since $\phi(\theta) \leq \frac{\theta^2}{2(1-\theta/3)}$ and by assumption,  $\sum_i \Lambda^i(S_i) \leq v$, we obtain
\begin{align*}
    &\mathbb P \left( \theta \sum_i (N^i(S_i) - \Lambda^i(S_i))  - \phi(\theta) \sum_i \Lambda^i(S_i) \geq x \right) \\
    &=\mathbb P \left(  \sum_i N^i(S_i) - \Lambda^i(S_i) \geq \theta^{-1} (x  + \phi(\theta) \sum_i \Lambda^i(S_i)) \right)  \\
    &\geq  \mathbb P \left(  \sum_i N^i(S_i) - \Lambda^i(S_i) \geq \theta^{-1} x  + \frac{\theta}{2(1-\theta/3)} v\right),
\end{align*}
and therefore,
\begin{align}
    \mathbb P \left(  \sum_i N^i(S_i) - \Lambda^i(S_i) \geq \theta^{-1} x  + \frac{\theta}{2(1-\theta/3)} v\right) \leq e^{-x}. \label{eq:bound-with-theta}
\end{align}


\paragraph{Step 3.} Note that the bound in \eqref{eq:bound-with-theta} is valid for any $\theta > 0$. 
But for any $\theta \in (0,3)$, we have
\begin{align*}
    \theta^{-1} (x  + \phi(\theta) v) \geq \sqrt{2vx} + \frac{x}{3}. 
\end{align*}
Therefore, together with \eqref{eq:bound-with-theta} we can conclude that
\begin{align*}
    \mathbb P \left(  \sum_i N^i(S_i) - \Lambda^i(S_i) \geq \sqrt{2vx} + \frac{x}{3} \right) \leq e^{-x}. 
\end{align*}

\end{proof}

\begin{lemma}\label{lem:bernstein1}
Let $N = (N^i)_{i=1, \dots, n}$ $n$ i.i.d. spatio-temporal Hawkes point processes on $[0,1] \times [0,1]^{d}$ with parameter $f$ satisfying Assumption \ref{ass:boundedness}. Let $(S_i)_{i=1, \dots, n} \subset [0,1]^{d+1}$ possibly random independent subsets. Then for any $x > 0$, there exists $\sigma^2,  b > 0$ constants independent of $n$ and $(S_i)_i$ such that
\begin{align}
    &\mathbb P\left( \frac{1}{n} \sum_{i=1}^n N^i(S_i) - \Lambda^i(S_i) \geq x \right) \leq e^{-\frac{nx^2}{2( \sigma^2 + bx)}}, \label{eq:bernstein} \\
    &\mathbb P\left( \frac{1}{n} \sum_{i=1}^n N^i(S_i) - \Lambda^i(S_i) \leq -x \right) \leq e^{-\frac{nx^2}{2(\sigma^2 +  bx)}}.\label{eq:bernstein2}
\end{align}
\end{lemma}

\begin{proof}
   This proof is based on the moment version of the standard Bernstein inequality, together with the fact that $N^i$ admits exponential moments, extending a result from \cite{hansen2015lasso}.
   
   
   We first recall the standard Bernstein inequalities with the moment assumption: let $(Z_i)_i$  i.i.d. and centered random variables and $b, \sigma^2 > 0$ such that for any $i \in [n]$,
    \begin{align*}
        &\mathbb E [Z_i^2] \leq \sigma^2 \\
        &\mathbb E [Z_i^k] \leq \frac{1}{2} k! b^{k-2}\sigma^2, \quad k \geq 2.
    \end{align*}
   Then it holds that
    \begin{align}
        &\mathbb P \left(\frac{1}{n} \sum_i Z_i  \geq x \right) \leq  \exp \left( - \frac{n x^2}{2(\sigma^2 + bx)}\right) \label{eq:bernstein-right}\\
        &\mathbb P \left(\frac{1}{n} \sum_i Z_i \leq -x \right) \leq \exp \left( - \frac{n x^2}{2(\sigma^2 + bx)}\right).\label{eq:bernstein-left}
    \end{align}
    Our goal is to apply the previous inequalities to  $Z_i := N^i(S_i) - \Lambda^i(S_i)$. By definition of $\Lambda^i$, $\mathbb E[N^i(S_i) - \Lambda^i(S_i)] = 0$. Moreover since $S_i \subset [0,1] \times [0,1]^d$ and using Proposition 2 in \cite{hansen2015lasso},  there exist constant $\theta, C  > 0$ that can only depend on $f$ such that
    \begin{align*}
        \mathbb E[e^{\theta_i N^i(S_i)}] \leq \mathbb E[e^{\theta N^i([0,1]^{d+1})}] \leq C,
    \end{align*}
    which, since $ \Lambda(S_i) \geq 0$, implies that
        \begin{align*}
        \mathbb E[e^{\theta_i (N^i(S_i) - \Lambda(S_i))}] 
        \leq C,
    \end{align*}
    meaning that $N^i(S_i) - \Lambda(S_i)$ admits exponential moments. From this we can deduce that
    \begin{align*}
        &\mathbb E [Z_i^2] \leq \frac{2}{\theta} \mathbb E[e^{\theta N^i(S_i)}] \leq \frac{2C}{\theta} 
        = : \sigma^2 \\
        &\mathbb E [Z_i^k] \leq \frac{k!}{\theta^k} \mathbb E[e^{\theta N^i(S_i)}] \leq \frac{k!C}{\theta^k} \leq  \frac{1}{2} k! \sigma^2 \frac{1}{\theta^{k-2}} =: \frac{1}{2} k! \sigma^2 b^{k-2},
    \end{align*}
    with $b := \frac{1}{\theta}$. Therefore, applying the Bernstein's inequality \eqref{eq:bernstein-right}, we obtain 
\begin{align*}
    \mathbb P\left( \frac{1}{n} \sum_{i=1}^n N^i(S_i) - \Lambda^i(S_i)\geq x \right) \leq e^{-\frac{nx}{2( \sigma^2 + bx)}}.
\end{align*}
    Similarly, applying \eqref{eq:bernstein-right} we obtain the left tail probability in \eqref{eq:bernstein2}.
\end{proof}

\section{Proofs of other results}\label{app:proof-lemmas}

\subsection{Proof of Lemma \ref{lem:event}}

\begin{lemma}[High probability event]
For any $\alpha > 0$, there exists $\delta_0, c_\alpha > 0$ such that
\begin{align*}
    \Omega_n := \left \{ \frac{\underline{\mu}}{1 - \|g_0\|_1}  - \delta_n \leq \frac{1}{n}\sum_i N^i[0,1]^{d+1}  \leq \frac{\bar \mu}{1 - \|g_0\|_1} + \delta_n \right\} \cap \left \{ \sup_{i=1,\dots, n} N^i[0,1]^{d+1} \leq c_\alpha \log n \right \},
\end{align*}
with $\delta_n = \delta_0\frac{\log n }{\sqrt{n}}$ and $c > 0$. Under Assumption \ref{ass:boundedness},
\begin{align*}
    \mathbb P_0[\Omega_n] \geq 1 - 3 n^{-\alpha}.
\end{align*}
    
\end{lemma}

\begin{proof}
    Let us define:
    \begin{align*}
        \Omega_1 := \left \{ \sup_{i=1,\dots, n} N^i[0,1]^{d+1} \leq c_\alpha \log n \right \}.
    \end{align*}
    We first prove that for any $\alpha > 0$, there exists $c_\alpha$ such that $\mathbb P_0(\Omega_1^c) \leq n^{-\alpha}$. First note that under our Assumption \ref{ass:boundedness}, each spatio-temporal process  $N^i_{t,s}$ is stochastically dominated by a temporal and stationary point process  $\bar N^i_t$ with intensity
    \begin{align*}
        \bar \lambda^i_t(f_0) = \bar \mu + \int_{t-a}^{t-} \bar g_0(t-u)d \bar N^i_{u},
    \end{align*}
    where $\bar g_0(u) = \int g_0(u,s)ds$.
    Thus, using Proposition 2.1 from \cite{reynaud2007some}, there exists $c> 0$ such that
    \begin{align*}
        \mathbb P_0(\Omega_1^c) = \mathbb P_0(\sup_{i=1,\dots, n} N^i[0,1]^{d+1} > c_\alpha \log n) \leq \mathbb P_0(\sup_{i=1,\dots, n} \bar N^i[0,1] > c_\alpha \log n) \leq e^{-c c_{\alpha} \log n} = n^{-c c_\alpha} \leq n^{-\alpha},
    \end{align*}
    if $c_\alpha> \alpha / c$.

    We now define 
    \begin{align*}
        \Omega_2 := \left \{ \frac{\underline{\mu}}{1 - \|g_0\|_1}  - \delta_n \leq \frac{1}{n}\sum_i N^i[0,1]^{d+1}  \leq \frac{\bar \mu}{1 - \|g_0\|_1} + \delta_n \right\},
    \end{align*}
    and prove that there exists $\delta_0$ such that $ \mathbb P_0[\Omega_2^c] \leq 2 n^{-\alpha}$. We use again the fact that each $N^i_{t,s}$ is stochastically dominated by $\bar N^i_t$ and stochastically dominates $\underline{N}^i_t$ with intensity
    \begin{align*}
         \underline{ \lambda}^i_t(f_0) = \underline{\mu} + \int_{t-a}^{t-} g_0(t-s)d \underline{ N}^i_{s}.
    \end{align*}
    Therefore,
    \begin{align*}
      \frac{\underline{\mu}}{1 - \|g_0\|_1} = \mathbb E_0[\underline{N}^i[0,1]] \leq \mathbb E_0[N^i[0,1]^{d+1}] \leq \mathbb E_0[\bar N^i[0,1]] = \frac{\bar \mu}{1 - \|g_0\|_1},
    \end{align*}
    and  using Lemma \ref{lem:bernstein1} with $S = [0,1]$,
    \begin{align*}
        \mathbb P_0 \left(\frac{1}{n}\sum_i N^i[0,1]^{d+1} \geq  \frac{\bar \mu}{1 - \|g_0\|_1} + \delta_n \right) &\leq \mathbb P_0 \left(\frac{1}{n}\sum_i \bar N^i[0,1] \geq  \frac{\bar \mu}{1 - \|g_0\|_1} + \delta_n \right) \\
        &= \mathbb P_0 \left( \frac{1}{n}\sum_i \bar N^i[0,1] - \mathbb E_0[\bar N^i[0,1]] \geq \delta_n \right) \\
        &\leq e^{- \frac{n\delta_n^2}{2(\sigma^2 + b\delta_n)}} \leq e^{-\delta_0 \log n /(4\sigma^2)} = n^{-\delta_0/(4 \sigma^2)} \leq n^{-\alpha}
    \end{align*}
   for $n$ large enough and $\delta_0 \geq 4 \sigma^2 \alpha$. Similarly,
   \begin{align*}
        \mathbb P_0 \left(\frac{1}{n}\sum_i N^i[0,1]^{d+1} \leq  \frac{\underline{ \mu}}{1 - \|g_0\|_1} - \delta_n \right) &\leq \mathbb P_0 \left(\frac{1}{n}\sum_i \underline{N}^i[0,1] \leq  \frac{\underline{\mu}}{1 - \|g_0\|_1} - \delta_n \right) \\
        &= \mathbb P_0 \left(\frac{1}{n}\sum_i \underline{N}^i[0,1] - \mathbb E_0[\underline{N}^i[0,1]] \leq - \delta_n \right) \\
        &\leq e^{- \frac{n\delta_n^2}{2(\sigma^2 + b\delta_n)}} \leq e^{-\delta_0 \log n /(4\sigma^2)} = n^{-\delta_0/(4 \sigma^2)} \leq n^{-\alpha}.
    \end{align*}
Thus we can conclude that $\mathbb P_0 (\Omega_2^c) \leq 2n^{-\alpha}$ which leads to
\begin{align*}
    \mathbb P_0 (\Omega_n^c) \leq \mathbb P_0 (\Omega_1^c) + \mathbb P_0 (\Omega_2^c) \leq 3 n^{-\alpha}.
\end{align*}
\end{proof}

\subsection{Proof of Lemma \ref{lem:KL}}

\begin{lemma}[Kullback-Leibler]
    Under Assumption \ref{ass:boundedness} and if $\bar \epsilon_n = o((\log n)^{-2})$, there exist $b_1,b_2 > 0$ such that for any $f \in B_\infty(\bar \epsilon_n)$,
    \begin{align}
        &KL(f,f_0) := \mathbb E_0[\log L(N | f_0) - \log L(N|f)] \leq \kappa n \bar \epsilon_n^2 ( 1 + o(1)) \label{eq:bound-kl}\\
        &\mathbb P_0( \log L(N | f_0) - \log L(N|f) > b_1 n \bar  \epsilon_n^2) \leq \frac{b_2}{n \bar  \epsilon_n^2},
    \end{align}
with
\begin{align}
    \kappa := \frac{4 \log 2}{\underline{\mu}}\left \{ 2 + 4\left(\frac{\bar \mu}{1 - \|g_0\|_1} + \Lambda_{0,2} \right) \right \}.
\end{align}
\end{lemma}

\begin{proof}

We first prove the first statement \eqref{eq:bound-kl}. This proof is organised in 3 main steps. In the first step, we re-write the KL divergence as a single integral over the true conditional intensity measure and decompose it into two terms, considering the high-probability event $\Omega_n$ and its complement $\Omega_n^c$. In the second and third steps, we control each of these terms. The first term can be controlled by the squared $\ell_2$-distance between the conditional intensities $\lambda_{t,s}^1(f)$ and $\lambda_{t,s}^1(f_0)$, which itself can be controlled by the squared $\ell_2$-distance on the parameter, i.e,  $\|f - f_0\|_2^2$. We then control the second term using the Cauchy-Schwarz inequality and a bound on the fourth-moment of $N$. In the following, when we compute integrals over the spatio-temporal domain $[0,1] \times [0,1]^d$, we omit the bounds in the integral   for ease of notation. We also note that since the domain is the $(d+1)$-dimensional hypercube, we use multiple times that $\int 1dtds = 1$.

Let $f \in B_\infty(f_0, \epsilon_T)$. We have
\begin{align*}
    KL(f,f_0) &= \mathbb{E}_0[\log L(N | f_0) - \log L(N|f)] \\
    &= \mathbb E_0 \left[ \sum_i \int - \log \frac{\lambda^i_{t,s}(f)}{\lambda^i_{t,s}(f_0)}dN_{t,s}^i + \int (\lambda^i_{t,s}(f) - \lambda^i_{t,s}(f_0))dt ds \right] \\
    &= n \mathbb E_0 \left[\int - \log \frac{\lambda^1_{t,s}(f)}{\lambda^1_{t,s}(f_0)}\lambda^1_{t,s}(f_0)dtds + \int (\lambda^1_{t,s}(f) - \lambda^1_{t,s}(f_0))dt ds \right] \\
    &=n \mathbb E_0 \left[ \int \psi(\frac{\lambda^1_{t,s}(f)}{\lambda^1_{t,s}(f_0)}) \lambda^1_{t,s}(f_0)  dt ds\right] \\
    &=n \underbrace{\mathbb E_0 \left[ \mathds{1}_{\Omega_n} \int \psi(\frac{\lambda^1_{t,s}(f)}{\lambda^1_{t,s}(f_0)}) \lambda^1_{t,s}(f_0)  dt ds\right]}_{=:I_1} + n \underbrace{\mathbb E_0 \left[ \mathds{1}_{\Omega_n^c} \int \psi(\frac{\lambda^1_{t,s}(f)}{\lambda^1_{t,s}(f_0)}) \lambda^1_{t,s}(f_0)  dt ds\right]}_{=:I_2},
\end{align*}
where in the third equality we have defined $\psi(x) = -\log x +  x -1 \geq 0, x > 0$.

\paragraph{Bound on $I_1$} We use that  $\psi(x) \leq - 4 \log (r) (x-1)^2$ for any $x \geq r$ and $r > 0$. We thus find a lower bound $r$ on the ratio $\frac{\lambda^1_{t,s}(f)}{\lambda^1_{t,s}(f_0)}$ on the event $\Omega_n$. First note that
\begin{align*}
        \frac{\lambda^1_{t,s}(f)}{\lambda^1_{t,s}(f_0)} &= 1 + \frac{\lambda^1_{t,s}(f) - \lambda^1_{t,s}(f_0)}{\lambda^1_{t,s}(f_0)} \geq 1 - \frac{|\lambda^1_{t,s}(f) - \lambda^1_{t,s}(f_0)|}{\lambda^1_{t,s}(f_0)}.
\end{align*}
We upper bound $\frac{|\lambda^1_{t,s}(f) - \lambda^1_{t,s}(f_0)|}{\lambda^1_{t,s}(f_0)}$ using Assumption \ref{ass:boundedness} and on $\Omega_n$. Since $f \in B_\infty(f_0, \bar \epsilon_n)$ and $f_0$ verifies Assumption \ref{ass:boundedness},
\begin{align*}
    \mu(t,s) \geq \mu_0(t,s)   - \|\mu - \mu_0\|_\infty \geq \underline{\mu} - \bar  \epsilon_n \geq \underline{\mu}/2,
\end{align*}
for any $t,s$ and any $n$ large enough, which implies that
\begin{align}\label{eq:lower-bound-lambda}
    \lambda^1_{t,s}(f) \geq \mu(t,s) \geq \underline{\mu}/2.
\end{align}
Similarly we have
\begin{align}\label{eq:upper-bound-mu}
    \mu(t,s) \leq \mu_0(t,s) + \|\mu - \mu_0\|_\infty \leq \bar \mu + \bar  \epsilon_n \leq \bar \mu + 1.
\end{align}
Moreover, on $\Omega_n$, $\sup_i \sup_{t \in [0,1]} N^i[t,t-A] \leq c_\alpha \log n$, therefore,
\begin{align}
   \left| \lambda^1_{t,s}(f_0) - \lambda^1_{t,s}(f) \right| &= \left|\mu_0(t,s) - \mu(t,s) + \sum_{t_i^1 \leq t} (g_0 - g)(t-t_i,s-s_i) \right| \nonumber \\
   &\leq \|\mu_0 - \mu \|_\infty + \|g_0 - g\|_\infty N^1(t,t-a) \nonumber  \\
   &\leq \|\mu_0 - \mu \|_\infty + c_\alpha \|g_0 - g\|_\infty \log n. \label{eq:ub-delta-lambda}
\end{align}
Since  $f \in B_\infty(f_0, \bar  \epsilon_n)$, this implies that 
\begin{align*}
    \frac{|\lambda^1_{t,s}(f) - \lambda^1_{t,s}(f_0)|}{\lambda^1_{t,s}(f_0)} &\leq \frac{\|\mu_0 - \mu \|_\infty + c_\alpha \|g_0 - g\|_\infty \log n}{\underline{\mu}} \leq \frac{\epsilon_n ( 1 + c_\alpha \log n)}{\underline{\mu}} \leq \frac{1}{2},\end{align*}
 for $n$ large enough, using that by assumption $ \bar \epsilon_n = o((\log n)^{-1})$, and thus,
\begin{align}
    \frac{\lambda^1_{t,s}(f)}{\lambda^1_{t,s}(f_0)} \geq \frac{1}{2}. \label{eq:lb-ratio-on}
\end{align}

Thus, with $r = \frac{1}{2}$, we obtain that $\psi(\frac{\lambda^1_{t,s}(f)}{\lambda^1_{t,s}(f_0)}) \leq  4 \log (2) (\frac{\lambda^1_{t,s}(f)}{\lambda^1_{t,s}(f_0)}-1)^2$ which leads to 
\begin{align*}
    I_1 = \mathbb E_0 \left[\mathds{1}_{\Omega_n} \int \psi(\frac{\lambda^1_{t,s}(f)}{\lambda^1_{t,s}(f_0)}) \lambda^1_{t,s}(f_0)  dt ds\right]
     &\leq 4 \log (2) \mathbb E_0 \left[ \int \left(\frac{\lambda^1_{t,s}(f)}{\lambda^1_{t,s}(f_0)} - 1\right)^2  \lambda^1_{t,s}(f_0) dt ds\right] \\
    &\leq \frac{4\log 2}{\underline{\mu}} \mathbb E_0 \left[ \int (\lambda^1_{t,s}(f) - \lambda^1_{t,s}(f_0))^2  dt ds\right], 
\end{align*}
under Assumption \ref{ass:boundedness}. Using that $(x + y)^2 \leq 2 x^2 + 2y^2$, we have
\begin{align}
    &\mathbb E_0 \left[ \int (\lambda^1_{t,s}(f) - \lambda^1_{t,s}(f_0))^2  dt ds\right] \nonumber \\
    & \leq 2 \|\mu - \mu_0\|_2^2 + 2 \mathbb E_0 \left[ \int \left( \sum_{t_i^1 < t}g(t-t_i^1, s-s_i^1) - g_0(t-t_i^1, s-s_i^1)  \right)^2 dt ds\right] \nonumber \\
    &= 2 \|\mu - \mu_0\|_2^2 + 2 \mathbb E_0 \left[ \int \left( \int_{u : u \in [t-a,t)} \int_{v: \|v-s\| \leq b }(g(t-u, s-v) - g_0(t-u, s-v)) dN^1_{u,v}  \right)^2 dt ds\right] \nonumber \\
    &\leq 2 \|\mu - \mu_0\|_2^2 \\
    & \qquad + 4 \mathbb E_0 \left[ \int \left( \int_{u : u \in [t-a,t)} \int_{v: \|v-s\| \leq b }(g(t-u, s-v) - g_0(t-u, s-v)) (dN^1_{u,v} - \lambda^1_{u,v}(f_0) du dv) \right)^2 dt ds\right] \nonumber \\
    & \qquad + 4 \mathbb E_0 \left[ \int \left(\int_{u : u \in [t-a,t)} \int_{v: \|v-s\| \leq b }(g(t-u, s-v) - g_0(t-u, s-v))  \lambda^1_{u,v}(f_0) du dv \right)^2 dt ds\right] \label{eq:bound-delta-lambda-squared}
\end{align}
To bound the second and third terms in the RHS of \eqref{eq:bound-delta-lambda-squared}, we will use the following identity (see, e.g., Theorem B12 in \cite{karr2017point}): for any deterministic and squared integrable function $h$,
\begin{align}
    &\mathbb E_0\left[ \int \left(\int_{u : u \in [t-a,t)} \int_{v: \|v-s\| \leq b } h(t-u,s-v) (dN^1_{u,v} - \lambda^1_{u,v}(f_0) du dv) \right)^2 dt ds \right] \nonumber \\
    &=  \mathbb E_0\left[ \int\int_{u : u \in [t-a,t)} \int_{v: \|v-s\| \leq b } h^2(t-u,s-v) \lambda^1_{u,v}(f_0)  du dv dt ds \right] \nonumber \\
    &= \int\int_{u : u \in [t-a,t)} \int_{v: \|v-s\| \leq b } h^2(t-u,s-v) \mathbb E_0\left[\lambda^1_{u,v}(f_0)\right]  du dv dt ds  \label{eq:square-martingale}.
\end{align}
We will also use an upper bound on $\max_{u,v} \mathbb E_0\left[\lambda^1_{u,v}(f_0)\right]$. Note that
\begin{align*}
   \mathbb E_0[ \lambda^1_{t,s}(f_0)] &= \mu_0(t,s) + \mathbb E_0\left[\int_{u : u \in [t-a,t)} \int_{v: \|v-s\| \leq b } g_0(t- u,s - v) d N^1_{u,v}\right] \\
   &=\mu_0(t,s) + \mathbb E_0\left[\int_{u : u \in [t-a,t)} \int_{v: \|v-s\| \leq b } g_0(t- u,s - v) \lambda^1_{u,v}(f_0) du dv\right] \\
   &\leq \bar \mu + \|g_0\|_1 \max_{u,v} \mathbb E_0\left[\lambda^1_{u,v}(f_0)\right],
\end{align*}
which implies
\begin{align}
    &\max_{t,s} \mathbb E_0\left[\lambda^1_{t,s}(f_0)\right] \leq \frac{\bar \mu}{1 - \|g_0\|_1} < \infty, \label{eq:upper-bound-elambda}
\end{align}
under Assumption \ref{ass:boundedness}.
Using \eqref{eq:square-martingale}  with $h = g - g_0$ and \eqref{eq:upper-bound-elambda}, we obtain
\begin{align*}
    &\mathbb E_0 \left[ \int \left( \int_{u : u \in [t-a,t)} \int_{v: \|v-s\| \leq b }(g(t-u, s-v) - g_0(t-u, s-v)) (dN^1_{u,v} - \lambda^1_{u,v}(f_0) du dv) \right)^2 dt ds\right]\\
    &\qquad \leq \| g-g_0 \|_2^2 \frac{\bar \mu}{1 - \|g_0\|_1}.
\end{align*}
Moreover by Cauchy-Schwarz inequality,
\begin{align*}
    &\mathbb E_0\left[ \int \left(\int (g-g_0)(t-u,s-v) \lambda^1_{u,v}(f_0) du dv \right)^2 dt ds \right] \\ &\leq \mathbb E_0\left[ \int \left( \int (g-g_0)^2(t-u,s-v)  dv du \int (\lambda^1_{u,v}(f_0))^2  dv du \right) dt ds \right] \\
    &\leq \|g-g_0\|_2^2  \mathbb E_0\left[ \int \int  (\lambda^1_{u,v}(f_0))^2  dv du dt ds \right] = \|g-g_0\|_2^2  \Lambda_{0,2},
\end{align*}
with
\begin{align*}
    \Lambda_{0,2} :=  \mathbb E_0\left[ \int  (\lambda^1_{u,v}(f_0))^2  dv du \right].
\end{align*}
We claim that $\Lambda_{0,2}  < \infty$, since
\begin{align*}
      \Lambda_{0,2} &\leq 2 \|\mu_0\|_2^2 + 2  \|g_0\|_2^2 \mathbb E_0\left[ (N[0,1]^{d+1})^2\right ],
\end{align*}
  $\|\mu_0\|_2^2 < \infty$ under Assumption \ref{ass:boundedness}, and $\mathbb E_0\left[ (N[0,1]^{d+1})^2\right ] < \infty$. The existence of the second moments of $N[0,1]^{d+1}$ comes from the fact that a spatio-temporal point process can be seen as a marked temporal point process (TPP) and any non-explosive TPP admits exponential moments on a finite domain, which implies that
\begin{align}
    \mathbb E_0 \left[ (N^1[0,1]^{d+1})^k \right] < \infty, \quad k > 0. \label{eq:finite-moments}
\end{align}

Hence, given \eqref{eq:bound-delta-lambda-squared},  we obtain
\begin{align}
    \mathbb E_0 \left[ \int (\lambda^1_{t,x}(f) - \lambda^1_{t,x}(f_0))^2  dt dx\right] &\leq 2 \|\mu - \mu_0\|_2^2 + 4\left(\frac{\bar \mu}{1 - \|g_0\|_1} + \Lambda_{0,2} \right)  \| g - g_0 \|_2^2 \nonumber \\
    &\leq \left \{ 2 + 4\left(\frac{\bar \mu}{1 - \|g_0\|_1} + \Lambda_{0,2} \right) \right \} \| f -f_0\|_2^2, \label{eq:d2T}
\end{align}    
and thus,
\begin{align*}
    I_1 &\leq  \frac{4\log 2}{\underline{\mu}}\left \{ 2 + 4\left(\frac{\bar \mu}{1 - \|g_0\|_1} + \Lambda_{0,2} \right) \right \}  \| f -f_0\|_2^2 \leq \frac{ 4\log 2}{\underline{\mu}}\left \{ 2 + 4\left(\frac{\bar \mu}{1 - \|g_0\|_1} + \Lambda_{0,2} \right) \right \} \bar  \epsilon_n^2 = \kappa \bar  \epsilon_n^2, 
\end{align*}
using that since $f \in B_\infty(f_0, \bar  \epsilon_n)$, $\| f -f_0\|_2^2 \leq \| f -f_0\|_\infty^2 \leq \bar  \epsilon_n^2$ and with
\begin{align*}
    \kappa = \frac{4 \log 2}{\underline{\mu}}\left \{ 2 + 4\left(\frac{\bar \mu}{1 - \|g_0\|_1} + \Lambda_{0,2} \right) \right \}.
\end{align*}

We now prove that $I_2 = o( \bar \epsilon_n^2)$. One one hand, under Assumption \ref{ass:boundedness} and using \eqref{eq:upper-bound-mu}, 
\begin{align}
     \frac{\lambda^1_{t,s}(f_0)}{\lambda^1_{t,s}(f)} &= 1 + \frac{\lambda^1_{t,s}(f_0) - \lambda^1_{t,s}(f)}{\lambda^1_{t,s}(f)} \leq 1 + \frac{|\lambda^1_{t,s}(f) - \lambda^1_{t,s}(f_0)|}{\lambda^1_{t,s}(f)} \leq 1 + 2\frac{\|\mu - \mu_0 \|_\infty + \| g - g_0\|_\infty N[0,1]^{d+1} }{\lambda^1_{t,s}(f)} \nonumber \\
     &\leq 1 + \frac{2 \bar \epsilon_n( 1 + N[0,1]^{d+1})}{\lambda^1_{t,s}(f)} \leq 1 + \frac{2 \bar \epsilon_n( 1 + N[0,1]^{d+1})}{\underline{\mu}}, \label{eq:ub-ratio-n}
\end{align}
and on the other hand, 
\begin{align}
        \frac{\lambda^1_{t,s}(f)}{\lambda^1_{t,s}(f_0)} 
        \leq 1 + \frac{ \bar \epsilon_n( 1 + N[0,1]^{d+1})}{\lambda^1_{t,s}(f_0)} \leq 1 + \frac{ \bar \epsilon_n( 1 + N[0,1]^{d+1})}{\underline{\mu}}.\label{eq:ub-ratio-n2} 
\end{align}
Thus, using that $\log x \leq x-1$,
\begin{align*}
    I_2 &= \mathbb E_0 \left[ \mathds{1}_{\Omega_n^c} \int  \left \{  - \log(\frac{\lambda^1_{t,s}(f)}{\lambda^1_{t,s}(f_0)}) +\frac{\lambda^1_{t,s}(f)}{\lambda^1_{t,s}(f_0)} - 1 \right \}\lambda^1_{t,s}(f_0)  dt ds\right] \\
    &= \mathbb E_0 \left[ \mathds{1}_{\Omega_n^c} \int  \left \{  \log(\frac{\lambda^1_{t,s}(f_0)}{\lambda^1_{t,s}(f)}) +\frac{\lambda^1_{t,s}(f)}{\lambda^1_{t,s}(f_0)} - 1 \right \}\lambda^1_{t,s}(f_0)  dt ds\right] \\ 
    &\leq \mathbb E_0 \left[ \mathds{1}_{\Omega_n^c} \int  \left \{   \log(1 + \frac{2\epsilon_n( 1 + N^1[0,1]^{d+1})}{\underline{\mu}}) + \frac{\epsilon_n( 1 + N^1[0,1]^{d+1})}{\underline{\mu}}  \right \}\lambda^1_{t,s}(f_0)  dt ds\right] \\
    &\leq \mathbb E_0 \left[ \mathds{1}_{\Omega_n^c} \int  \left \{   \frac{2\epsilon_n( 1 + N^1[0,1]^{d+1})}{\underline{\mu}}+ \frac{\epsilon_n( 1 + N^1[0,1]^{d+1})}{\underline{\mu}}  \right \}\lambda^1_{t,s}(f_0)  dt ds\right] \\
    &= \frac{3}{\bar \mu} \bar \epsilon_n \mathbb E_0 \left[ \mathds{1}_{\Omega_n^c} \int   \lambda^1_{t,s}(f_0)  dt ds \right ] + \frac{3}{\bar \mu}  \bar \epsilon_n \mathbb E_0 \left[ \mathds{1}_{\Omega_n^c} N^1[0,1]^{d+1}  \int \lambda^1_{t,s}(f_0)  dt ds \right] \\
    &\leq C \bar \epsilon_n \sqrt{ \mathbb P_0(\Omega_n^c)\mathbb E_0 \left[\int   (\lambda^1_{t,s}(f_0))^2  dt ds \right ] } + C \bar \epsilon_n \sqrt{ \mathbb P_0(\Omega_n^c) \mathbb E_0 \left[ (N^1[0,1]^{d+1})^2 \int   (\lambda^1_{t,s}(f_0)^2  dt ds \right ]}\\
    &\leq C \bar \epsilon_n \sqrt{ \mathbb P_0(\Omega_n^c)\mathbb E_0 \left[\int   (\lambda^1_{t,s}(f_0))^2  dt ds \right ] } + C\bar  \epsilon_n \sqrt{ \mathbb P_0(\Omega_n^c) \sqrt{ \mathbb E_0 \left[ (N^1[0,1]^{d+1})^4 \right] \mathbb E_0 \left[\int   (\lambda^1_{t,s}(f_0))^4  dt ds\right]}}
\end{align*}
with $C = \frac{3}{\bar \mu}$ and using Cauchy-Schwarz inequality in the last two inequalities. Moreover, using Lemma \ref{lem:event}, $\mathbb P_0(\Omega_n^c) \leq n^{-\alpha}$ for any $\alpha > 0$. 
From \eqref{eq:finite-moments}, we have $\mathbb E_0 \left[ (N^1[0,1]^{d+1})^4 \right] < \infty$ and thus,
\begin{align*}
   \mathbb E_0 \left[ (\lambda^1_{t,s}(f_0))^4 \right]\leq 8 \mu_0(t,s)^4 + 8\|g_0\|_\infty^4 \mathbb E_0 \left[(N^1[0,1]^{d+1})^4 \right] < \infty.
\end{align*}
Thus, $I_2 = O(n^{-\alpha/2}) = o(\bar  \epsilon_n^2)$ for any $\alpha > 1$. We therefore conclude that for $n$ large enough,
\begin{align*}
    KL(f,f_0) = n(I_1 + I_2) \leq \kappa n \bar  \epsilon_n^2 ( 1 + o(1)),
\end{align*}
which proves the first statement of Lemma \ref{lem:KL}.

We now prove the second statement. The proof relies on bounding the variance of $\log L(f_0|N) - \log L(f|N)$,  and applying Chebyshev's inequality. To bound the variance of $\log L(f_0|N) - \log L(f|N)$, we will decompose it on $\Omega_n$ and on $\Omega_n^c$. Recall that
\begin{align*}
    \log L(f_0|N) - \log L(f|N) = \sum_{i=1}^n\int \log \frac{\lambda^i_{t,s}(f_0)}{\lambda^i_{t,s}(f)}dN_{t,s}^i + \int (\lambda^i_{t,s}(f) - \lambda^i_{t,s}(f_0))dt ds.
\end{align*}
For any $i \in [n]$, let 
\begin{align*}
    Z_i :=\int \log \frac{\lambda^i_{t,s}(f_0)}{\lambda^i_{t,s}(f)}dN_{t,s}^i + \int (\lambda^i_{t,s}(f) - \lambda^i_{t,s}(f_0))dt ds.
\end{align*}
Note that $\mathbb E_0[Z_i] = KL(f,f_0)$. We have
\begin{align*}
    &\mathbb E_0[Z_i^2] = \mathbb E_0 \left[\left( \int \log \frac{\lambda^i_{t,s}(f_0)}{\lambda^i_{t,s}(f)}dN_{t,s}^i + \int (\lambda^i_{t,s}(f) - \lambda^i_{t,s}(f_0))dt ds\right)^2 \right] \\
    &= \mathbb E_0 \left[\left( \int \log \frac{\lambda^i_{t,s}(f_0)}{\lambda^i_{t,s}(f)}\lambda^i_{t,s}(f_0)dt ds + \int \log \frac{\lambda^i_{t,s}(f_0)}{\lambda^i_{t,s}(f)}(dN_{t,s}^i -\lambda^i_{t,s}(f_0)dt ds) + \int (\lambda^i_{t,s}(f) - \lambda^i_{t,s}(f_0))dt ds\right)^2 \right] \\
    &= \mathbb E_0 \left[\left( \int \psi \left( \frac{\lambda^i_{t,s}(f)}{\lambda^i_{t,s}(f_0)} \right)\lambda^i_{t,s}(f_0)dt ds + \int \log \frac{\lambda^i_{t,s}(f)}{\lambda^i_{t,s}(f_0)}(dN_{t,s}^i - \lambda^i_{t,s}(f_0))dt ds)\right)^2 \right] \\
     &\leq 2 \mathbb E_0 \left[\left( \int \psi \left( \frac{\lambda^i_{t,s}(f)}{\lambda^i_{t,s}(f_0)} \right)\lambda^i_{t,s}(f_0)dt ds \right)^2 \right] +2 \mathbb E_0 \left[\left(  \int \log \frac{\lambda^i_{t,s}(f)}{\lambda^i_{t,s}(f_0)}(dN_{t,s}^i - \lambda^i_{t,s}(f_0))dt ds)\right)^2 \right] \\
     &\leq  2 \mathbb E_0 \left[ \int \psi \left( \frac{\lambda^i_{t,s}(f)}{\lambda^i_{t,s}(f_0)} \right)^2 (\lambda^i_{t,s}(f_0))^2 dt ds  \right] +2 \mathbb E_0 \left[ \int (\log \frac{\lambda^i_{t,s}(f)}{\lambda^i_{t,s}(f_0)})^2 \lambda^i_{t,s}(f_0)dt ds \right].
\end{align*}
using Cauchy-Schwarz inequality and \eqref{eq:square-martingale} in the last inequality. We first bound
\begin{align*}
    \mathbb E_0 \left[ \mathds{1}_{\Omega_n} \int \psi \left( \frac{\lambda^i_{t,s}(f)}{\lambda^i_{t,s}(f_0)} \right)^2 (\lambda^i_{t,s}(f_0))^2 dt ds  \right] \quad \text{ and } \quad  \mathbb E_0 \left[ \mathds{1}_{\Omega_n} \int (\log \frac{\lambda^i_{t,s}(f)}{\lambda^i_{t,s}(f_0)})^2 \lambda^i_{t,s}(f_0)dt ds \right]
\end{align*}
Recall from \eqref{eq:lb-ratio-on} that on $\Omega_n$ and for $n$ large enough, $ \frac{\lambda^1_{t,s}(f)}{\lambda^1_{t,s}(f_0)} \geq \frac{1}{2}$.
Thus, using again that $\psi(x) \leq - \log (r)(x-1)^2$ for any $x \geq r > 0$ with $r = \frac{1}{2}$ and under Assumption \ref{ass:boundedness}, we obtain
\begin{align*}
    \psi \left(\frac{\lambda^i_{t,s}(f)}{\lambda^i_{t,s}(f_0)} \right)^2 (\lambda^i_{t,s}(f_0))^2 \leq (\log 2)^2 \frac{(\lambda^i_{t,s}(f_0) - \lambda^i_{t,s}(f))^4}{\lambda^i_{t,s}(f_0)^2} \leq \frac{(\log 2)^2 }{\underline{\mu}^2}(\lambda^i_{t,s}(f_0) - \lambda^i_{t,s}(f))^4.
\end{align*}
Thus, 
\begin{align*}
    \mathbb E_0 \left[ \mathds{1}_{\Omega_n} \int \psi \left( \frac{\lambda^i_{t,s}(f)}{\lambda^i_{t,s}(f_0)} \right)^2 (\lambda^i_{t,s}(f_0))^2 dt ds  \right] &\leq \frac{(\log 2)^2 }{\underline{\mu}^2} \mathbb E_0 \left[ \mathds{1}_{\Omega_n} \int (\lambda^i_{t,s}(f_0) - \lambda^i_{t,s}(f))^4 dt ds  \right] \\
    &\leq  \frac{(\log 2)^2 }{\underline{\mu}^2} \bar \epsilon_n^4 ( 1+ c_\alpha \log n)^4 = o(\bar  \epsilon_n^2),
\end{align*}
using \eqref{eq:ub-delta-lambda} in the last inequality and that by assumption, $\bar  \epsilon_n = o((\log n)^{-2})$. Moreover, using that $|\log(x)| \leq - 2 \log (r)|x-1|$ for any $x \geq r > 0$ and with $r = \frac{1}{2}$, we also obtain
\begin{align*}
    (\log \frac{\lambda^i_{t,s}(f)}{\lambda^i_{t,s}(f_0)})^2 \lambda^i_{t,s}(f_0) \leq (2 \log 2)^2 \frac{(\lambda^i_{t,s}(f_0) - \lambda^i_{t,s}(f))^2}{\lambda^i_{t,s}(f_0)} \leq \frac{(2 \log 2)^2}{\underline{\mu}} (\lambda^i_{t,s}(f_0) - \lambda^i_{t,s}(f))^2, 
\end{align*}
which implies using \eqref{eq:d2T} that
\begin{align*}
    \mathbb E_0 \left[ \mathds{1}_{\Omega_n}\int (\log \frac{\lambda^i_{t,s}(f)}{\lambda^i_{t,s}(f_0)})^2 \lambda^i_{t,s}(f_0)dt ds \right] \leq \frac{(2 \log 2)^2}{\underline{\mu}}\mathbb E_0 \left[\int (\lambda^i_{t,s}(f_0) - \lambda^i_{t,s}(f))^2 dt ds \right] \leq (\log 2) \kappa \bar \epsilon_n^2.
\end{align*}
We now bound the remaining terms
\begin{align*}
    \mathbb E_0 \left[ \mathds{1}_{\Omega_n^c} \int \psi \left( \frac{\lambda^i_{t,s}(f)}{\lambda^i_{t,s}(f_0)} \right)^2 (\lambda^i_{t,s}(f_0))^2 dt ds  \right] \quad \text{ and } \quad  \mathbb E_0 \left[ \mathds{1}_{\Omega_n^c} \int (\log \frac{\lambda^i_{t,s}(f)}{\lambda^i_{t,s}(f_0)})^2 \lambda^i_{t,s}(f_0)dt ds \right].
\end{align*}
Using \eqref{eq:ub-ratio-n} and \eqref{eq:ub-ratio-n2}, we have
\begin{align*}
    \left| \log  \frac{\lambda^i_{t,s}(f)}{\lambda^i_{t,s}(f_0)}  \right| &=  \log  \frac{\lambda^i_{t,s}(f)}{\lambda^1_{t,s}(f_0)} \vee  \log  \frac{\lambda^i_{t,s}(f_0)}{\lambda^1_{t,s}(f)}
    \leq \log \left( 1 + \frac{\bar  \epsilon_n( 1 + N[0,1]^{d+1})}{\lambda^i_{t,s}(f_0) \wedge \lambda^i_{t,s}(f)} \right) \leq \frac{\bar \epsilon_n( 1 + N[0,1]^{d+1})}{\lambda^i_{t,s}(f_0) \wedge \lambda^i_{t,s}(f)}.
\end{align*}
Thus, we obtain
\begin{align}
&\mathbb E_0 \left[ \mathds{1}_{\Omega_n^c} \int \left(\log \frac{\lambda^i_{t,s}(f)}{\lambda^i_{t,s}(f_0)}\right)^2 \lambda^i_{t,s}(f_0)dt ds \right] \leq \mathbb E_0 \left[ \mathds{1}_{\Omega_n^c} \bar \epsilon_n^2 ( 1 + N[0,1]^{d+1})^2 (\lambda^1_{t,s}(f_0))^{-1} (1 \vee \frac{\lambda^1_{t,s}(f_0)}{\lambda^1_{t,s}(f)} )^2\right]  \nonumber \\
    &\leq \frac{4}{\underline{\mu}} \bar \epsilon_n^2 \mathbb E_0 \left[ \mathds{1}_{\Omega_n^c}  (N[0,1]^{d+1})^2  (1 + \frac{2\bar \epsilon_n( 1 + N[0,1]^{d+1})}{\underline{\mu}} )^2\right] \nonumber \\
    &\lesssim \bar \epsilon_n^2 \sqrt{\mathbb P_0(\Omega_n^c) \mathbb E_0 \left[(N[0,1]^{d+1})^8\right]} = o(\bar \epsilon_n^2), \label{eq:ub-o23}
\end{align}
using  again \eqref{eq:finite-moments} and that $\mathbb P_0(\Omega_n^c) = o(1)$. Similarly, using that $(a+b)^2 \leq 2a^2 + 2b^2$,
\begin{align*}
 &\mathbb E_0 \left[ \mathds{1}_{\Omega_n^c} \int \psi \left( \frac{\lambda^i_{t,s}(f)}{\lambda^i_{t,s}(f_0)} \right)^2 (\lambda^i_{t,s}(f_0))^2 dt ds  \right] \\
 &\leq 2 \mathbb E_0 \left[ \mathds{1}_{\Omega_n^c} \int \log \left( \frac{\lambda^i_{t,s}(f)}{\lambda^i_{t,s}(f_0)} \right)^2 (\lambda^i_{t,s}(f_0))^2 dt ds  \right]   +  2 \mathbb E_0 \left[ \mathds{1}_{\Omega_n^c} \int(\lambda^i_{t,s}(f) - \lambda^i_{t,s}(f_0))^2 dt ds  \right]  \\
 &\lesssim \mathbb E_0 \left[ \mathds{1}_{\Omega_n^c} \bar  \epsilon_n^2 ( 1 + N[0,1]^{d+1})^2    (1 + \frac{\bar \epsilon_n( 1 + N[0,1]^{d+1})}{\underline{\mu}} )\right] \\
 &\qquad + \sqrt{\mathbb P_0 (\mathds{1}_{\Omega_n^c}) \mathbb E_0 \left[ \left(\int(\lambda^i_{t,s}(f) - \lambda^i_{t,s}(f_0))^2 dt ds \right)^2 \right] } = o(\bar  \epsilon_n^2),
\end{align*}
where in the last equality we have used \eqref{eq:ub-o23} for the first term on the RHS and that
\begin{align*}
    \mathbb E_0 \left[ \left(\int(\lambda^i_{t,s}(f) - \lambda^i_{t,s}(f_0))^2 dt ds \right)^2 \right] &\leq \mathbb E_0 \left[ \int(\lambda^i_{t,s}(f) - \lambda^i_{t,s}(f_0))^4 dt ds  \right] \\
    &\leq 8 \|\mu-\mu_0\|_\infty^4 + 8 \|g-g_0\|_\infty^4  \mathbb E_0 \left[ (N^i[0,1]^{d+1})^4 \right] \lesssim \bar \epsilon_n^4 = o(\bar \epsilon_n^2),
\end{align*}
using \eqref{eq:ub-delta-lambda}.
We can thus conclude that
\begin{align*}
    \mathbb Var[Z_i^2] = \mathbb E_0[Z_i^2] - \mathbb E_0^2[Z_i] \leq \mathbb E_0[Z_i^2] \leq \epsilon_n^2 (\kappa (\log 2)  + o(1)) \leq 2 \kappa \epsilon_n^2,
\end{align*}
for $n$ large enough. We then apply Chebychev's inequality: for any $x > 0$,
\begin{align*}
    \mathbb P_0\left[ \log L(f_0|N) - \log L(f|N) -  KL(f,f_0) > x \right] \leq \frac{n\mathbb E_0[Z_i^2]}{x^2}
\end{align*}
Thus, for any $\epsilon > 0$ and $n$ large enough,
\begin{align*}
    \mathbb P_0\left[ \log L(f_0|N) - \log L(f|N) > \kappa ( 1+ \epsilon) n \bar \epsilon_n^2 + x  \right] \leq \frac{2\kappa n \bar \epsilon_n^2}{x^2},
\end{align*}
and with $x = x_1 n \bar \epsilon_n^2$ with $x_1 > 0$, we obtain for any $\epsilon > 0$ and $n$ large enough,
\begin{align*}
    &\mathbb P_0\left[ \log L(f_0|N) - \log L(f|N) > ( \kappa( 1 + \bar \epsilon) + x_1) n  \bar \epsilon_n^2   \right] \leq \frac{2\kappa}{x_1^2 n \bar \epsilon_n^2} \\
    &\iff \mathbb P_0\left[ \log L(f_0|N) - \log L(f|N) > b_1 n \bar \epsilon_n^2   \right] \leq \frac{b_2}{n \bar \epsilon_n^2},
\end{align*}
with $b_1 = \kappa ( 1 + \epsilon ) + x_1 > \kappa$ and $b_2 = \frac{2\kappa}{x_1^2}$. By choosing $x_1 = \kappa \epsilon $, for some fixed $\epsilon > 0$, we obtain $b_1 = \kappa ( 1 + 2\epsilon ), b_2 =  \frac{2}{\kappa \epsilon^2}$, and this terminates the proof of this lemma.

\end{proof}

\subsection{Proof of Lemma \ref{lem:tests}}

\begin{lemma}[Tests]
Under Assumptions \ref{ass:sieves} and \ref{ass:boundedness}, there exists a test function  $\phi := \phi(N, \epsilon_n)$ such that
\begin{align*}
    &\mathbb E_0[\phi \mathds{1}_{\Omega_n}] = o(1) \\
    &\sup_{f \in \mathcal{F}_n}  \mathbb E_f[(1 - \phi)  \mathds 1_{\Omega_n} \mathds 1_{f \in A_n}] \leq e^{-b_2n \epsilon_n^2}
\end{align*}
where $b_2 > c_1$, $A_n$ is defined in \eqref{eq:def-an} and $\Omega_n$ is defined in Lemma \ref{lem:event}.
\end{lemma}

\begin{proof}
The main idea is to construct individual test functions $\phi(f_1) = \phi(f_1,N, \epsilon_n)$ for testing a parameter $f_1 \in A_n$ against $f_0$:
\begin{align}\label{eq:test-f1}
    \phi(f_1) := \mathds 1_{\frac{1}{n} \sum_{i=1}^n N^i(S_{1,i}) - \Lambda^{i,0}(S_{1,i}) > v_n} \vee \mathds 1_{\frac{1}{n} \sum_{i=1}^n N^i(S_{1,i}^c) - \Lambda^{i,0}(S_{1,i}^c) > v_n}
\end{align}
with $S_{1,i} := \{ (t,s) \in [0,1]^{d+1} : \lambda^i_{f_1}(t,s) \geq \lambda^i_{f_0}(t,s) \}$, $v_n > 0$ a sequence and
\begin{align*}
    \Lambda^{i,0}(S_{1,i}) = \int_{S_{1,i}} \lambda^i_{f_0}(t,s) dt ds.
\end{align*}
Note that $\Lambda^{i,0}(S_{1,i})$ is the compensator of $N^i$ on $S_{1,i}$ under $\mathbb P_{0}$, thus,  $\frac{1}{n}\sum_{i}N^i(S_{1,i}) - \Lambda^{i,0}(S_{1,i})$ and $\frac{1}{n}\sum_{i} N^i(S_{1,i}^c) - \Lambda^{i,0}(S_{1,i}^c)$ concentrate around 0 under $\mathbb P_{0}$. Therefore, intuitively, under $\mathbb P_{0}$ (the null), $\phi(f_1)$ should go to 0 provided that $v_n$ is not too small. Under $\mathbb P_{f_1}$ (the alternative), then on average, $\lambda^i_{f_1}(t,s)$ is either mostly greater (case 1) or mostly smaller (case 2) than $\lambda^i_{0}(t,s)$. In case 1, $|S_{1,i}| > |S_{1,i}^c|$, and note that $\Lambda^{i,1}(S_{1,i}) > \Lambda^{i,0}(S_{1,i})$. Therefore, under $\mathbb P_{f_1}$, in this case $\frac{1}{n}\sum_{i}N^i(S_{1,i}) - \Lambda^{i,0}(S_{1,i}) = \frac{1}{n}\sum_{i}N^i(S_{1,i}) - \Lambda^{i,1}(S_{1,i}) + \Lambda^{i,1}(S_{1,i}) - \Lambda^{i,0}(S_{1,i})$ would concentrate on  $\lim_{n \to \infty}\frac{1}{n}\sum_{i} \Lambda^{i,1}(S_{1,i}) - \Lambda^{i,0}(S_{1,i})> 0$, which implies that $\phi(f_1)$ should go to 1. Similar reasoning can be applied to case 2 by considering $S_{1,i}^c$ instead of $S_{1,i}$.

To control the type-I and type-II error of our test on the event $\Omega_n$, we use a Bernstein  concentration inequality using an adaptation of Proposition 2 in  \cite{hansen2015lasso} for the spatio-temporal context stated in Lemma \ref{lem:bernstein1}:
\begin{align*}
    \mathbb P_0\left( \frac{1}{n} \sum_{i=1}^n N^i(S_i) - \Lambda^i(S_i) \geq x \right) \leq e^{-\frac{nx^2}{2(\bar \sigma^2 + \bar bx)}},
\end{align*}
for any $x > 0$ and with $\bar \sigma, \bar b$ constants independent of the subsets $(S_{i})_i$.

For the type-I error $\mathbb E_0[\phi]$, we will leverage a minimal $L_1$-covering net of $\mathcal{F}_n$  (defined in Assumption \ref{ass:sieves}) by balls of radius 
$\zeta j \epsilon_n$ with $\zeta > 0$ a constant which value will be fixed later, denoted by $\mathcal{N}_j$. Then, under Assumption \ref{ass:sieves}, the cardinal of  $\mathcal{N}_j$ (i.e., the covering number) is bounded by
\begin{align*}
   |\mathcal{N}_j| = C(\zeta j \epsilon_n, \mathcal{F}_j, \|\cdot \|_1) \leq
   C(\zeta_0  \epsilon_n, \mathcal{F}_n, \|\cdot \|_1) =: |\mathcal{N}_0| \leq e^{c_3 n \epsilon_n^2},
\end{align*}
if $j \geq \zeta_0/\zeta$ and this holds if $M \geq \zeta_0/\zeta$, and with $\mathcal{N}_0$ a mimimal $L_1$-covering net  of $\mathcal{F}_n$ by balls of radius $\zeta_0 \epsilon_n$ (recall that $\zeta_0$ is defined in Assumption \ref{ass:sieves}). 
For any $f_j \in \mathcal{N}_j$, we define a test function $\phi(f_j)$ as in \eqref{eq:test-f1} with a sequence $v_n$ that depends on $j$ and is defined below. Then, we define our global test function as
\begin{align*}
    \phi = \max_{j \geq M} \max_{f_j \in \mathcal{N}_j}\phi(f_j).
\end{align*}
For the type-II error $\sup_{f \in \mathcal{F}_n}\mathbb E_f[(1-\phi) \mathds{1}_{f \in A_n}]$,
we adopt a slicing approach of $A_n$ (defined in \eqref{eq:def-an}). We define for any $j \geq M$, the ``slice'' 
\begin{align*}
     \mathcal{F}_j := \{ f \in \mathcal{F}_n: j \epsilon_n \leq d_S(f, f_0) \leq (j+1) \epsilon_n\},
\end{align*}
so that we can re-express $A_n$ as (assuming wlog that $M$ is an integer)
\begin{align*}
    A_n = \bigcup_{j=M}^{\infty} \mathcal{F}_j.
\end{align*}
Since $\mathcal{F}_j \subseteq \mathcal{F}_n \subseteq \cup_{f_j \in \mathcal{N}_j} \{ f \in  \mathcal{F}_n : \|f - f_1\|_1 \leq \zeta j \epsilon_n \}$,  therefore
\begin{align*}
    \sup_{f \in \mathcal{F}_n}\mathbb E_f[(1-\phi) \mathds{1}_{f \in A_n}] &\leq \sum_{j\geq M} \sup_{f \in \mathcal{F}_n}\mathbb E_f[(1-\phi) \mathds{1}_{f \in \mathcal{F}_j}] \\
    &\leq  \sum_{j\geq M} \sup_{f_j \in \mathcal{N}_j} \sup_{f \in \mathcal{F}_n : \|f - f_j\|_1\leq \zeta j \epsilon_n} \mathbb E_f[(1-\phi(f_j)) \mathds{1}_{f \in \mathcal{F}_j}]
\end{align*}


We now specify the sequence $v_n$ in $\phi(f_j)$ and decompose $\phi(f_j)$  into two sub-tests, i.e., for any $f_{j} \in \mathcal{N}_j$, we define
\begin{align*}
    &\phi(f_j) = \phi^+(f_j) \vee \phi^-(f_j) \\
    &\phi^+(f_j) = \mathds 1_{\frac{1}{n} \sum_{i=1}^n N^i(S_{1,i}) - \Lambda^{i,0}(S_{1,i}) > v_n} \\
    &\phi^-(f_j) = \mathds 1_{\frac{1}{n} \sum_{i=1}^n N^i(S_{1,i}^c) - \Lambda^{i,0}(S_{1,i}^c) > v_n}.
\end{align*}
We define $v_n = x_1j\epsilon_n$ ($x_1>0$ a constant which value will be fixed later) and
 apply Lemma \ref{lem:bernstein1} with $x = x_1 j \epsilon_n = v_n$, $x_1 > 0$, and $S_i = S_{1,i}$:
\begin{align*}
    \mathbb E_0[\phi^+(f_j)] = \mathbb P_0\left[\frac{1}{n} \sum_{i=1}^n N^i(S_{1,i}) - \Lambda^{i,0}(S_{1,i}) > x_1 j \epsilon_n \right] \leq e^{- \frac{x_1^2 n j^2 \epsilon_n^2}{2(\bar \sigma^2 + \bar bx_1 j \epsilon_n)}}.
\end{align*}
We can apply the same inequality with $S_i = S_{1,i}^c$ and obtain that for the test function
\begin{align*}
        \phi^-(f_j) = \mathds 1_{\frac{1}{n} \sum_{i=1}^n N^i(S_{1,i}^c) - \Lambda^{i,0}(S_{1,i}^c) > v_n},
\end{align*}
that
\begin{align*}
    \mathbb E_0[ \phi^-(f_j)] \leq e^{- \frac{x_1^2 n j^2 \epsilon_n^2}{2(\bar \sigma^2 + \bar bx_1 j \epsilon_n)}}.
\end{align*}
Thus, 
we obtain:
\begin{align*}
    \mathbb E_0[ \phi(f_j)] = \mathbb E_0[ \phi^+(f_j) \vee  \phi^-(f_j)]\ \leq \mathbb E_0[ \phi^+(f_j)] + \mathbb E_0[ \phi^-(f_j)] \leq 2e^{- \frac{x_1^2 n j^2 \epsilon_n^2}{2(\bar \sigma^2 + \bar bx_1 j \epsilon_n)}}.
\end{align*}

We distinguish 2 cases:
\begin{itemize}
    \item \textbf{$j \epsilon_n > \bar \sigma^2 / (\bar b x_1)$}. Then the RHS above is
    \begin{align*}
        \leq 2 e^{- \frac{x_1^2 n j^2 \epsilon_n^2}{4\bar bx_1 j \epsilon_n}} = 2e^{- \frac{x_1 n j \epsilon_n}{4\bar b}}.
    \end{align*}
    \item \textbf{$j \epsilon_n \leq \bar \sigma^2 / (\bar b x_1)$}. Then the RHS above is
    \begin{align*}
        \leq  2e^{- \frac{x_1^2 n j^2 \epsilon_n^2}{4\bar \sigma^2}}.
    \end{align*}
\end{itemize}

Recall that our global test function is defined as
$$\phi : = \max_{j \geq M} \max_{f_1 \in \mathcal{N}_j} \phi(f_j) \leq \sum_{j\geq M} \sum_{f_1 \in \mathcal{N}_j} \phi(f_j).$$
Note that the number of terms of the second sum in the RHS of the previous inequality is bounded by $|\mathcal{N}_0|$.

\paragraph{Type I error.} We can  upper-bound $\mathbb E_0[\phi]$ by
\begin{align*}
   \mathbb E_0[\phi] &\leq \sum_{j \geq M} \sum_{f_1 \in \mathcal{N}_j} \mathbb E_0[\phi_{1,j}] \leq
   \sum_{j > \bar \sigma^2 / (\bar b x_1 \epsilon_n)}  2|\mathcal{N}_0| e^{- \frac{x_1 n j \epsilon_n}{4\bar b}} + \sum_{M \leq j \leq \bar \sigma^2 / (\bar b x_1 \epsilon_n)}  2|\mathcal{N}_0|  e^{- \frac{x_1^2 n j^2 \epsilon_n^2}{4\bar \sigma^2}} \\
   &\leq \sum_{j > \bar \sigma^2 / (\bar b x_1 \epsilon_n)} 2 e^{c_3 n \epsilon_n^2} e^{- \frac{x_1 n j \epsilon_n}{4\bar b}} + \sum_{M \leq j \leq  \bar \sigma^2 / (\bar b x_1 \epsilon_n)}2  e^{c_3 n \epsilon_n^2} e^{- \frac{x_1^2 n j^2 \epsilon_n^2}{4\bar \sigma^2}} \\
   &\leq 2  e^{c_3 n \epsilon_n^2} \sum_{j \geq  M} e^{- \frac{\min(x_1,x_1^2) n j \epsilon_n^2}{4 \max(b, \bar \sigma^2) }} \\
   &\leq 4 e^{- \frac{\min(x_1,x_1^2)  M n \epsilon_n^2}{8  \max(b, \bar \sigma^2)}} = o(1).
\end{align*}
where in the third inequality we use that $\max \left( e^{- \frac{x_1 n j \epsilon_n}{4\bar b}},  e^{- \frac{x_1^2 n j^2 \epsilon_n^2}{4\bar \sigma^2}} \right)\leq e^{- \frac{ \min(x_1, x_1^2) n j \epsilon_n^2}{4 \max(\bar b, \bar \sigma^2)}} $ and in the last inequality that $\frac{\min(x_1, x_1^2) M}{8 \max(\bar b, \bar \sigma^2)} > c_3$ for $M$ and $n$ large enough.

\paragraph{Type II error.}

We now bound $\sup_{f \in \mathcal{F}_n}  \mathbb E_f[(1 - \phi)  \mathds 1_{\Omega_n} \mathds 1_{f \in A_n}]$. First note that for any $j' \geq M$ and $f_1' \in \mathcal{N}_{j'}$,
\begin{align*}
    1-\phi = 1 -  \max_{j \geq M} \max_{f_j \in \mathcal{N}_j} \phi(f_j) \leq 1 -\phi(f_{j'}).
\end{align*}
Recall that since $A_n = \bigcup_{j \geq M} \mathcal{F}_j$, we have
\begin{align*}
    \sup_{f \in \mathcal{F}_n}  \mathbb E_f[(1 - \phi)  \mathds 1_{\Omega_n} \mathds 1_{f \in A_n}] &\leq \sum_{j \geq M} \sup_{f \in \mathcal{F}_n} \mathbb E_f[(1 - \phi) \mathds 1_{\Omega_n} \mathds 1_{f \in \mathcal{F}_j}] \\
    &\leq \sum_{j \geq M}  \sup_{f_j \in \mathcal{N}_j} \sup_{f \in \mathcal{F}_n, \|f-f_{1}\| \leq \zeta j\epsilon_n} \mathbb E_f[(1 - \phi(f_j)) \mathds 1_{\Omega_n} \mathds{1}_{f \in \mathcal{F}_j}],
\end{align*}
using that for any $j$, $\mathcal{F}_n \subset \bigcup_{f_j \in \mathcal{N}_j}  \{f \in \mathcal{F}_n : \|f - f_j\|_1 \leq \zeta j \epsilon_n\}$.
Let  $f_j \in  \mathcal{N}_j$ and $f \in \mathcal{F}_n$ such that $\|f-f_{j}\| \leq \zeta j\epsilon_n$. 
We have
\begin{align}
    \mathbb E_f[\mathds{1}_{f \in \mathcal{F}_j}\mathds 1_{\Omega_n}(1-\phi(f_j))] &=\mathbb E_f[\mathds{1}_{f \in \mathcal{F}_j}\mathds{1}_{\Omega_n}(1-\phi^+(f_j) \vee \phi^-(f_j))] \nonumber \\ 
    &\leq \mathbb E_f[\mathds{1}_{f \in \mathcal{F}_j}\mathds{1}_{\Omega_n}(1-\phi^+(f_j))] \wedge \mathbb E_f[\mathds{1}_{f \in \mathcal{F}_j} \mathds{1}_{\Omega_n}(1-\phi^-(f_j))] \label{eq:phi-phipm}\\ 
 \mathbb E_f[\mathds{1}_{f \in \mathcal{F}_j} \mathds{1}_{\Omega_n}(1-\phi^+(f_j))]   &=\mathbb E_f[\mathds{1}_{f \in \mathcal{F}_j} \mathds 1_{\Omega_n} \mathds 1_{\frac{1}{n} \sum_{i=1}^n N^i(S_{1,i}) - \Lambda^{i,0}(S_{1,i}) < v_n}] \nonumber \\
    &= \mathbb E_f[\mathds{1}_{f \in \mathcal{F}_j} \mathds 1_{\Omega_n} \mathds 1_{\frac{1}{n} \sum_{i=1}^n N^i(S_{1,i}) - \Lambda^{i,f}(S_{1,i}) + (\Lambda^{i,f}(S_{1,i}) - \Lambda^{i,0}(S_{1,i})) < v_n}], \nonumber
\end{align}
where $S_{1,i} := \{ (t,s) \in [0,1]^d : \lambda^i_{f_j}(t,s) \geq \lambda^i_{f_0}(t,s) \}$, $v_n > 0$.
We can lower-bound $\frac{1}{n} \sum_{i} \Lambda^{i,f}(S_{1,i}) - \Lambda^{i,0}(S_{1,i})$ by
\begin{align}
    \frac{1}{n} \sum_{i} \Lambda^{i,f}(S_{1,i}) - \Lambda^{i,0}(S_{1,i}) &= \frac{1}{n} \sum_{i} \Lambda^{i,f}(S_{1,i}) - \Lambda^{i,f_1}(S_{1,i}) + \Lambda^{i,f_1}(S_{1,i}) -\Lambda^{i,0}(S_{1,i}) \nonumber \\
    &=\frac{1}{n} \sum_{i} \int_{S_{1,i}} (\lambda^i_{t,s}(f) - \lambda^i_{t,s}(f_j) + \lambda^i_{t,s}(f_j) - \lambda^i_{t,s}(f_0))dtds \nonumber \\
    &\geq - \frac{1}{n} \sum_{i} \int_{S_{1,i}} |\lambda^i_{t,s}(f) - \lambda^i_{t,s}(f_j)| dt ds  + \frac{1}{n} \sum_{i} \int_{S_{1,i}} (\lambda^i_{t,s}(f_j) - \lambda^i_{t,s}(f_0))dtds \nonumber \\
    &\geq - d_S(f,f_j) +  \frac{1}{n} \sum_{i}\int_{S_{1,i}} (\lambda^i_{t,s}(f_j) - \lambda^i_{t,s}(f_0))dtds. \label{eq:lb-Lambda}
\end{align}
 Using Lemma \ref{lem:bound-ds-by-f1}, on $\Omega_n$, for any $f \in \mathcal{F}_j$,
\begin{align}\label{eq:l1-bounds-ds}
    d_{S}(f,f_j) \leq N_0 \|f-f_j\|_1 \leq N_0 \zeta j\epsilon_n.
\end{align}
Let $f \in \mathcal{F}_j$. We now consider two cases:
\begin{itemize}
    \item \textbf{Case 1}:
    \begin{align*}
    &\frac{1}{n} \sum_{i} \int_{S_{1,i}} (\lambda^i_{t,s}(f_j) - \lambda^i_{t,s}(f_0))dtds > \frac{1}{n} \sum_{i} \int_{S_{1,i}^c} (\lambda^i_{t,s}(f_0) - \lambda^i_{t,s}(f_j))dtds.
\end{align*}
    \item \textbf{Case 2}:
    \begin{align*}
    &\frac{1}{n} \sum_{i} \int_{S_{1,i}} (\lambda^i_{t,s}(f_j) - \lambda^i_{t,s}(f_0))dtds \leq \frac{1}{n} \sum_{i} \int_{S_{1,i}^c} (\lambda^i_{t,s}(f_0) - \lambda^i_{t,s}(f_j))dtds.
\end{align*}
\end{itemize}
Assume first that \textbf{Case 1} holds. Then
\begin{align*}
\frac{1}{n} \sum_{i} \int_{S_{1,i}} (\lambda^i_{t,s}(f_j) - \lambda^i_{t,s}(f_0))dtds > \frac{d_S(f_0,f_j)}{2}.
\end{align*}
Moreover, 
\begin{align*}
    d_S(f_0,f_j) = \frac{1}{n} \sum_{i=1}^n \int |\lambda_{t,s}^i(f_j) - \lambda_{t,s}^i(f_0)|dtds 
    &\geq \frac{1}{n} \sum_{i=1}^n \int |\lambda_{t,s}^i(f) - \lambda_{t,s}^i(f_0)|dtds \\
    &-\frac{1}{n} \sum_{i=1}^n \int |\lambda_{t,s}^i(f_j) - \lambda_{t,s}^i(f)|dt ds \\
    &\geq j\epsilon_n - \frac{1}{n} \sum_{i=1}^n  \int |\lambda_{t,s}^i(f_j) - \lambda_{t,s}^i(f)| dt ds = j\epsilon_n - d_S(f,f_j),
\end{align*}
since $f \in \mathcal{F}_j$. Besides, using Lemma \ref{lem:bound-ds-by-f1}, on $\Omega_n$, we have
\begin{align*}
    d_S(f,f_j) \leq N_0 \|f-f_j\|_1 \leq N_0 \zeta j \epsilon_n,
\end{align*}
therefore,
\begin{align*}
    d_S(f_0,f_j) \geq j\epsilon_n \left ( 1 - N_0 \zeta  \right ) > 0,
\end{align*}
if $\zeta < N_0^{-1}$.

From \eqref{eq:lb-Lambda} and \eqref{eq:l1-bounds-ds}, this implies 
\begin{align*}
    \frac{1}{n} \sum_{i} \Lambda^{i,f}(S_{1,i}) - \Lambda^{i,0}(S_{1,i}) &\geq - d_S(f,f_j) +\frac{d_S(f_0,f_j)}{2} \\
    &\geq - N_0 \zeta j \epsilon_n + \frac{j \epsilon_n (1- N_0 \zeta)}{2} \\
    &\geq (1/2 - 3N_0 \zeta/2) j \epsilon_n \geq j \epsilon_n/4,
\end{align*}
by choosing 
\begin{align*}
    (1/2 - 3N_0 \zeta/2) \geq \frac{1}{4} \iff \zeta \leq \frac{N_0}{6}.
\end{align*}
 Thus, with $v_n = nj \epsilon_n/8$, we obtain
\begin{align}
      \mathbb E_f[\mathds{1}_{f \in \mathcal{F}_j} \mathds{1}_{\Omega_n}(1-\phi_{1,j}^+)]  &\leq \mathbb E_f[\mathds{1}_{\Omega_n} \mathds 1_{\frac{1}{n} \sum_{i=1}^n N^i(S_{1,i}) - \Lambda^{i,f}(S_{1,i}) < v_n - j\epsilon_n/4}] \nonumber \\
      &= \mathbb E_f[\mathds{1}_{\Omega_n} \mathds 1_{\frac{1}{n} \sum_{i=1}^n N^i(S_{1,i}) - \Lambda^{i,f}(S_{1,i}) < (1/8 - 1/4) j\epsilon_n}] \nonumber \\
       &= \mathbb E_f[\mathds{1}_{\Omega_n} \mathds 1_{\frac{1}{n} \sum_{i=1}^n N^i(S_{1,i}) - \Lambda^{i,f}(S_{1,i}) < -j\epsilon_n/8}].\label{eq:upper-bound-ef}
\end{align}
To bound the RHS we use another form of Bernstein's inequality stated in Lemma \ref{lem:bernstein2}:
\begin{align*}
    \mathbb P_f\left(  \sum_{i=1}^n N^i(S_{1,i}) - \Lambda^i(S_{1,i}) \leq -  \sqrt{2vx} - \frac{x}{3}  \right) \leq e^{-x},
\end{align*}
for any $x > 0$ and with $v \geq \sum_i \int_{S_{1,i}} \lambda^i_{t,s}(f)dt ds = \sum_i \Lambda^{i,f}(S_{1,i})$. We first find such $v$. We have
\begin{align*}
    \sum_i \Lambda^{i,f}(S_{1,i}) = \sum_i \int_{S_{1,i}} \lambda^{i}_{t,s}(f) dt ds \leq  \sum_i \int_{} \lambda^{i}_{t,s}(f) dt ds \leq n \|\mu\|_1 + \|g\|_1 \sum_i N^i[0,1]^{d+1}.
\end{align*}
Moreover since $f \in \mathcal{F}_j$ and on $\Omega_n$, using Lemma \ref{lem:sd-bound}, we have
\begin{align}
    \|\mu\|_1 + \|g\|_1 \frac{1}{n} \sum_i N^i[0,1]^{d+1} &\leq  \|\mu_0\|_1 + \|g_0\|_1 \frac{1}{n} \sum_i N^i[0,1]^{d+1} +  d_S(f,f_0) \nonumber \\
    &\leq \|\mu_0\|_1 + e_0 \|g_0\|_1 + \|g_0\|_1 + (j + 1)  \epsilon_n. \label{eq:bound-ds}
\end{align}
with $e_0 = \frac{\bar \mu}{1 - \|g_0\|_1}$. Thus, letting $C_0 =  \|\mu_0\|_1 +  \|g_0\|_1 (e_0 +1 )$, we obtain
\begin{align*}
    \sum_i \Lambda^{i,f}(S_{1,i}) \leq n C_0 + (j+1) n \epsilon_n =: v.
\end{align*}
Therefore,
\begin{align}
    \mathbb E_f[\mathds{1}_{f \in \mathcal{F}_j}  \mathds 1_{\Omega_n}(1-\phi^+(f_j))] &= \mathbb E_f[\mathds{1}_{\Omega_n} \mathds 1_{\frac{1}{n} \sum_{i=1}^n N^i(S_{1,i}) - \Lambda^{i,f}(S_{1,i}) < -\sqrt{2vx} - \frac{x}{3}}] \\
    &\leq \mathbb P_f\left[\frac{1}{n} \sum_{i=1}^n N^i(S_{1,i}) - \Lambda^{i,f}(S_{1,i}) < - \sqrt{2vx} - \frac{x}{3}\right] \leq  e^{-x}. \label{eq:bernstein4f}
\end{align}

We distinguish 2 cases:
\begin{itemize}
    \item \textbf{$j \epsilon_n \leq C_0 +1$}. Then we apply \eqref{eq:bernstein4f} with $x = x_1 n j^2\epsilon_n^2$ with
    $0< x_1\leq \frac{3}{56 (C_0 + 1)}$. In particular, $x_1(C_0 + 1) \leq 1$ thus $\sqrt{x_1(C_0 + 1)} \leq x_1(C_0 + 1)$. Then,
\begin{align*}
    \sqrt{2vx} + \frac{x}{3} &= \sqrt{2x_1vn }j\epsilon_n + \frac{x_1 n j^2\epsilon_n^2}{3} \\
    &=\left(\sqrt{2x_1 (C_0 + (j+1) \epsilon_n)}  + \frac{x_1  j\epsilon_n}{3} \right)n j\epsilon_n \\
    &\leq \left( 2\sqrt{x_1(C_0 + 1)} + \frac{x_1(C_0+1)}{3} \right) n j\epsilon_n \\
    &\leq \frac{7 x_1(C_0+1)}{3} n j\epsilon_n \leq \frac{n j\epsilon_n}{8}.
\end{align*}
    Thus, with \eqref{eq:bernstein4f}, we obtain that
    \begin{align*}
       \mathbb E_f[\mathds{1}_{\Omega_n} \mathds 1_{\frac{1}{n} \sum_{i=1}^n N^i(S_{1,i}) - \Lambda^{i,f}(S_{1,i}) < -j\epsilon_n/8}]  \leq \mathbb E_f[\mathds{1}_{\Omega_n} \mathds 1_{\sum_{i=1}^n N^i(S_{1,i}) - \Lambda^{i,f}(S_{1,i}) < -\sqrt{2vx} - \frac{x}{3}}] \leq e^{-n x_1 j^2 \epsilon_n^2}.
    \end{align*}

    \item \textbf{$j \epsilon_n > C_0 +1$}. Then we apply \eqref{eq:bernstein4f} with $x = x_0 n j\epsilon_n$ with
    $x_0 \leq \min( \frac{C_0 + 1}{32^2}, \frac{3}{16})$. In particular $2\sqrt{\frac{x_0}{C_0 + 1}} \leq \frac{1}{16}$. We then have
        \begin{align*}
    \sqrt{2vx} + \frac{x}{3} &= \sqrt{2x_0v n j\epsilon_n} + \frac{x_0 n j\epsilon_n}{3} \\
    &=\left(\frac{\sqrt{2x_0(C_0 + (j+1) \epsilon_n)}}{j\epsilon_n} + \frac{x_0}{3} \right)n j\epsilon_n \\
    &\leq \left( 2\frac{\sqrt{x_0}}{\sqrt{j\epsilon_n}} + \frac{x_0}{3} \right)n j\epsilon_n \\
    &\leq \left( 2\frac{\sqrt{x_0}}{\sqrt{C_0+1}} + \frac{x_0}{3} \right)j\epsilon_n \leq \frac{n j\epsilon_n}{8}.
\end{align*}
Thus, with \eqref{eq:upper-bound-ef} and \eqref{eq:bernstein4f}, we obtain that
    \begin{align*}
      \mathbb E_f[\mathds{1}_{f \in \mathcal{F}_j}  \mathds 1_{\Omega_n}(1-\phi^+(f_j))] &\leq \mathbb E_f[\mathds{1}_{\Omega_n} \mathds 1_{\frac{1}{n} \sum_{i=1}^n N^i(S_{1,i}) - \Lambda^{i,f}(S_{1,i}) < -j\epsilon_n/8}] \\
      &\leq \mathbb E_f[\mathds{1}_{\Omega_n} \mathds 1_{\sum_{i=1}^n N^i(S_{1,i}) - \Lambda^{i,f}(S_{1,i}) < -\sqrt{2vx} - \frac{x}{3}}] \leq e^{-x_0 n j \epsilon_n}.
    \end{align*}
\end{itemize}

Note that the above bounds are independent of $f$ and $f_1$. Therefore, we have proven that
\begin{align*}
    \sup_{f_1 \in \mathcal{N}_j} \sup_{f \in \mathcal{F}_n, \|f-f_{1}\| \leq \zeta j\epsilon_n} \mathbb E_f[(1 - \phi^+(f_j)) \mathds 1_{\Omega_n} \mathds{1}_{f \in \mathcal{F}_j}] \leq \begin{cases}
        e^{-x_1 n j^2 \epsilon_n^2} & \text{ if } j\epsilon_n \leq C_0 + 1 \\
        e^{-x_0 n j \epsilon_n} & \text{ if } j\epsilon_n > C_0 + 1
    \end{cases}.
\end{align*}

Now assume that \textbf{Case 2} holds. Then 
\begin{align*}
    \frac{1}{n} \sum_{i} \int_{S_{1,i}^c} (\lambda^i_{t,s}(f_1) - \lambda^i_{t,s}(f_0))dtds \geq \frac{d_S(f_0,f_1)}{2}.
\end{align*}
By applying the same computations with  $S_{1,i}$ replaced by $S_{1,i}^c$, we obtain
\begin{align*}
     \sup_{f_1 \in \mathcal{N}_j} \sup_{f \in \mathcal{F}_n, \|f-f_{1}\| \leq \zeta j \epsilon_n} \mathbb E_f[\mathds{1}_{f \in \mathcal{F}_j}  \mathds 1_{\Omega_n}(1- \phi^-(f_j))] \leq \begin{cases}
        e^{-x_1 n j^2 \epsilon_n^2} & \text{ if } j\epsilon_n \leq C_0 + 1 \\
        e^{-x_0 n j \epsilon_n} & \text{ if } j\epsilon_n > C_0 + 1
    \end{cases}
\end{align*}

Given \eqref{eq:phi-phipm}, overall this implies that
\begin{align*}
    \sum_{j \geq M}  \sup_{f_1 \in \mathcal{N}_j} \sup_{f \in \mathcal{F}_n, \|f-f_{1}\| \leq \zeta j\epsilon_n} \mathbb E_f[\mathds 1_{f \in \mathcal{F}_j}(1 -  \phi(f_j)) \mathds 1_{\Omega_n}] &\leq \sum_{M \leq j \leq (C_0 +1)\epsilon_n^{-1} } 
    e^{-  x_1^2 n j^2 \epsilon_n^2} + \sum_{j \geq (C_0 +1)\epsilon_n^2} 
    e^{-  x_0 n j \epsilon_n} \\
    &\leq  e^{-  \min(x_0, x_1) n M \epsilon_n^2/2} = e^{-  b_2 n \epsilon_n^2}.
\end{align*}
with $b_2 = \frac{\min(x_0, x_1)M}{2}$ and $b_2 > c_1$ for $M$ large enough.

\end{proof}

\subsection{Proof of Lemma \ref{lem:ef}}

\begin{lemma}
For any $f \in A_{n}^c$, there exists a constant $p_0 >0$ such that on $\Omega_n$,
    \begin{align*}
        \mathbb E_f[Z_1] \geq p_0 \|f-f_0\|_1.
    \end{align*}
\end{lemma}

\begin{proof}
    Recall that
        \begin{align*}
        Z_1 := \int_0^{t^1_2} \int_{[0,1]^{d}} |\lambda^i_{t,s}(f) -  \lambda^i_{t,s}(f_0)| dt ds,
    \end{align*}
    where $t^1_2$ is the time of the second event or $t_2^1=1$ if there are less than 2 events. Define the event
    \begin{align*}
        \Omega_0 := \{ N^1[0,1]^{d+1} = 0 \},
    \end{align*}
    i.e., $\Omega_0$ is the event that the sequence has no events.
    Then on $\Omega_0$, $t_2^1 = 1$ and
    \begin{align*}
        \mathbb E_f[Z_1 \mathds 1_{\Omega_0} ] =  \mathbb E_f\left[\mathds 1_{\Omega_0}\int_0^{1} \int_{[0,1]^{d}} |\mu(t,s) - \mu_0(t,s)| dt ds \right] = \|\mu - \mu_0\|_1 \mathbb P_f[\Omega_0].
    \end{align*}
     Moreover, with $\mathbb Q$ the measure of a homogeneous Poisson point process with unit intensity on $[0,1]^{d+1}$ we have
    \begin{align*}
        \mathbb E_f[\mathds 1_{\Omega_0 } ] = \mathbb E_{\mathbb Q}[\mathcal{L}_f\mathds 1_{\Omega_0} ],
    \end{align*}
    where $\mathcal{L}_f$ is the likelihood process on $[0,1]^{d+1}$ defined as 
    \begin{align*}
        \mathcal{L}_f = \exp \left(1 - \int \lambda_{t,s}^1(f) dt ds + \int \log (\lambda_{t,s}^1(f)) dN_{t,s} \right)
    \end{align*}
    On $\Omega_0$,
    \begin{align*}
        \mathcal{L}_f = \exp \left(1 - \int \mu(t,s) dt ds \right) \geq \exp \left(1 - \|\mu\|_1 \right).
    \end{align*}
    From \eqref{eq:ineq-ds}, on $\Omega_n$,
    \begin{align*}
        \|\mu\|_1 \leq \|\mu_0\|_1 + \frac{3 e_0}{2} \|g_0\|_1  +   d_S(f,f_0) \leq \|\mu_0\|_1 + \frac{3 e_0}{2} \|g_0\|_1 + M\epsilon_n
    \end{align*}
    since $f \in A_{n}^c$. Thus, for $n$ large enough, on $\Omega_0 \cap \Omega_n$,
    \begin{align*}
         \mathcal{L}_f \geq \exp \left(- \|\mu_0\|_1 -  \frac{3 e_0}{2} \|g_0\|_1  \right) =: \mathcal{L}_0,
    \end{align*}
    and thus,
    \begin{align}\label{eq:ef-omega0}
        \mathbb E_f[Z_1 \mathds 1_{\Omega_0 } ] \geq \mathbb E_f[Z_1 \mathds 1_{\Omega_0 \cap \Omega_n} ] \geq \mathcal{L}_0 \|\mu-\mu_0\|_1 \mathbb Q(\Omega_0).
    \end{align}
    Since $\mathbb Q$ is a homogeneous Poisson process with intensity one,
    \begin{align*}
        \mathbb Q(\Omega_0)= e^{-1}.
    \end{align*}
    Therefore, we can conclude that
    \begin{align*}
        \mathbb E_f[Z_1 \mathds 1_{\Omega_0 } ] \geq \mathcal{L}_0 e^{-1} \|\mu-\mu_0\|_1.
    \end{align*}
    Now let $\tau \in (0, 1-2a)$ and define the subspaces $S_\tau^- = [0,\tau] \times [0,1]^d$, $S_\tau = [\tau, \tau+a] \times [0,1]^d$, $S_\tau^+ = [\tau+a, \tau+2a] \times [0,1]^d$ and the event:
    \begin{align*}
        \Omega_\tau = \{ N^1(S_\tau^-) = 0, N^1(S_\tau) = 1, N^1(S_\tau^+) = 0 \}.
    \end{align*}
    Note that on $\Omega_\tau$, $t_2^1 > t_1^1 + a$.
    Then,
    \begin{align}
        \mathbb E_f[Z_1 \mathds 1_{\Omega_\tau} ] &\geq \mathbb E_f\left[\mathds 1_{\Omega_\tau} \int_0^{\tau} \int_{[0,1]^{d}} |\mu(t,s) - \mu_0(t,s)| dt ds \right] \nonumber \\
        &+\mathbb E_f\left[\mathds 1_{\Omega_\tau} \int_{t_1^1}^{t_1^1 + a} \int_{[0,1]^{d}} |\mu(t,s) + g(t-t_1^1,s - s_1^1) - \mu_0(t,s) - g_0(t-t_1^1,s - s_1^1)| dt ds \right] \nonumber \\
        &\geq \mathbb E_f\left[\mathds 1_{\Omega_\tau} \int_{t_1^1}^{t_1^1 + a} \int_{[0,1]^{d}} \left| |\mu(t,s) - \mu_0(t,s)| - |g(t-t_1^1,s - s_1^1) -g_0(t-t_1^1,s - s_1^1)| \right| dt ds \right]
    \end{align}
    Let first assume that
    \begin{align*}
        &\|\mu - \mu_0\|_1 \geq \frac{\mathcal{L}_0 ae^{-2a-\tau}}{2} \|g -g_0\|_1 \\
        &\implies \|\mu - \mu_0\|_1 + \frac{\mathcal{L}_0 ae^{-2a-\tau}}{2} \|\mu -\mu_0\|_1 \geq \frac{\mathcal{L}_0 ae^{-2a-\tau}}{2} \|f -f_0\|_1 \\
        &\implies  \|\mu - \mu_0\|_1 \geq \frac{\mathcal{L}_0 ae^{-2a-\tau}}{2 + \mathcal{L}_0 ae^{-2a-\tau}} \|f-f_0\|_1.
    \end{align*}
    Then from \eqref{eq:ef-omega0},
    \begin{align*}
        \mathbb E_f[Z_1] \geq \mathbb E_f[Z_1 \mathds 1_{\Omega_0 } ] \geq \frac{\mathcal{L}_0^2ae^{-2a-\tau-1}}{2 + \mathcal{L}_0 ae^{-2a-\tau}} \|f -f_0\|_1.
    \end{align*}
    In the alternative case where  
    \begin{align*}
            &\|\mu - \mu_0\|_1 < \frac{\mathcal{L}_0 ae^{-2a-\tau}}{2} \|g -g_0\|_1 \\
        &\implies \|f - f_0\|_1 < \|g -g_0\|_1 + \frac{\mathcal{L}_0 ae^{-2a-\tau}}{2} \|g -g_0\|_1 \\
        &\implies  \|g - g_0\|_1 \geq \frac{2}{2 + \mathcal{L}_0 ae^{-2a-\tau}} \|f-f_0\|_1,
    \end{align*}
    then
    \begin{align*}
        \mathbb E_f[Z_1 \mathds 1_{\Omega_\tau} ] &\geq \mathbb E_f\left[\mathds 1_{\Omega_\tau} \int_{t_1^1}^{t_1^1 + a} \int_{[0,1]^{d}}   |g(t-t_1^1,s - s_1^1) -g_0(t-t_1^1,s - s_1^1)| - |\mu(t,s) - \mu_0(t,s)| dt ds \right] \\
        &\geq  \mathbb E_f\left[\mathds 1_{\Omega_\tau} \int_{t_1^1}^{t_1^1 + a} \int_{[0,1]^{d}}   |g(t-t_1^1,s - s_1^1) -g_0(t-t_1^1,s - s_1^1)| dt ds \right] - \|\mu - \mu_0\|_1 \mathbb E_f\left[\mathds 1_{\Omega_\tau}\right].
    \end{align*}
    Moreover, 
    \begin{align*}
         &\mathbb E_f\left[\mathds 1_{\Omega_\tau} \int_{t_1^1}^{t_1^1 + a} \int_{[0,1]^{d}}   |g(t-t_1^1,s - s_1^1) -g_0(t-t_1^1,s - s_1^1)| dt ds \right]\\
         &=  \mathbb E_{\mathbb Q}\left[\mathcal{L}_f \mathds 1_{\Omega_\tau} \int_{t_1^1}^{t_1^1 + a} \int_{[0,1]^{d}}   |g(t-t_1^1,s - s_1^1) -g_0(t-t_1^1,s - s_1^1)| dt ds \right] \\
         &\geq \mathcal{L}_0 \mathbb E_{\mathbb Q}\left[\mathds 1_{\Omega_\tau} \int_{0}^{a} \int_{[0,1]^{d}}   |g(u,s - s_1^1) -g_0(u,s - s_1^1)| du ds \right] = \mathcal{L}_0 \mathbb E_{\mathbb Q}\left[\mathds 1_{\Omega_\tau}\right] \|g-g_0\|_1.
    \end{align*}
    We have
    \begin{align*}
        \mathbb E_{\mathbb Q}\left[\mathds 1_{\Omega_\tau}\right] &= \mathbb P_{\mathbb Q}\left[N^1(S_\tau^-) = 0\right] \mathbb P_{\mathbb Q}\left[N^1(S_\tau) = 1\right]\mathbb P_{\mathbb Q}\left[N^1(S_\tau^+) = 0\right] \\
        &= e^{-\tau} \times a e^{-a} \times e^{-a} = a e^{-2a - \tau}.
    \end{align*}
    Thus we obtain
    \begin{align*}
         \mathbb E_f[Z_1 \mathds 1_{\Omega_\tau} ] &\geq  \mathcal{L}_0 a e^{-2a - \tau}\|g-g_0\|_1 - \|\mu-\mu_0\|_1 \\
         &\geq \frac{\mathcal{L}_0 a e^{-2a - \tau}}{2}\|g-g_0\|_1 \\
         &\geq \frac{\mathcal{L}_0 a e^{-2a - \tau}}{2 + \mathcal{L}_0 a e^{-2a - \tau}} \|f-f_0\|_1.
    \end{align*}
    Since $\Omega_0$ and $\Omega_\tau$ are disjoint events, we can conclude that
    \begin{align*}
        \mathbb E_f[Z_1 ] &\geq \mathbb E_f[Z_1 \mathds 1_{\Omega_0} ] + \mathbb E_f[Z_1 \mathds 1_{\Omega_\tau} ] \\
        &\geq \min \left( \frac{\mathcal{L}_0 a e^{-2a - \tau}}{2 + \mathcal{L}_0 a e^{-2a - \tau}}, \frac{\mathcal{L}_0^2ae^{-2a-\tau-1}}{2 + \mathcal{L}_0 ae^{-2a-\tau}}\right) \|f-f_0\|_1 \\
        &= \frac{\mathcal{L}_0 a e^{-2a - \tau}}{2 + \mathcal{L}_0 a e^{-2a - \tau}} \min \left( 1, \mathcal{L}_0 e^{-1}\right) \|f-f_0\|_1 = \frac{\mathcal{L}_0^2ae^{-2a-\tau-1}}{2 + \mathcal{L}_0 ae^{-2a-\tau}} \|f-f_0\|_1,
    \end{align*}
    since $\mathcal{L}_0<1$. Therefore, with
    \begin{align}\label{eq:def-p0}
        p_0 :=\frac{\mathcal{L}_0^2ae^{-2a-\tau-1}}{2 + \mathcal{L}_0 ae^{-2a-\tau}},
    \end{align}
    for any $\tau \in (0,1-2a)$, we obtain the result of Lemma \ref{lem:ef}.
\end{proof}

\bibliographystyle{plainnat}
\bibliography{sample}

\end{document}